\documentclass[12pt, reqno]{amsart}
\usepackage[dvipdfmx]{graphicx}
\usepackage{ascmac}
\usepackage{color}
\usepackage{type1cm}
\usepackage{wrapfig}
\usepackage{comment}
\usepackage{ytableau}
\usepackage{braket}
\usepackage{latexsym}
\usepackage{mathtools}
\usepackage{amsmath}
\usepackage{amsthm}
\usepackage{mathrsfs}
\usepackage{lmodern}
\usepackage[T1]{fontenc}
\usepackage{textcomp}
\usepackage[utf8]{inputenc}
\usepackage{color}
\usepackage{amssymb}
\usepackage{enumerate}
\usepackage[abbrev]{amsrefs}
\usepackage{indentfirst}
\usepackage{bmpsize}
\usepackage{hyperref}
\usepackage{enumitem}
\usepackage{url}
\usepackage{autobreak}
\usepackage[justification=justified]{caption}
\usepackage{bm}
\usepackage{here}
\usepackage[labelformat=empty,subrefformat=parens]{subcaption}
\usepackage[pagewise]{lineno}
\usepackage{adjustbox}
\usepackage{longtable}
\usepackage{ltablex}
\usepackage{vtable,hhline}
\usepackage{subcaption}
\usepackage{caption}
\usepackage[subrefformat=parens]{subcaption}
\usepackage{graphbox} 
\usepackage{empheq}

\newcommand*\patchAmsMathEnvironmentForLineno[1]{
  \expandafter\let\csname old#1\expandafter\endcsname\csname #1\endcsname
  \expandafter\let\csname oldend#1\expandafter\endcsname\csname end#1\endcsname
  \renewenvironment{#1}
     {\linenomath\csname old#1\endcsname}
     {\csname oldend#1\endcsname\endlinenomath}}
\newcommand*\patchBothAmsMathEnvironmentsForLineno[1]{
  \patchAmsMathEnvironmentForLineno{#1}
  \patchAmsMathEnvironmentForLineno{#1*}}
\AtBeginDocument{
\patchBothAmsMathEnvironmentsForLineno{equation}
\patchBothAmsMathEnvironmentsForLineno{align}
\patchBothAmsMathEnvironmentsForLineno{flalign}
\patchBothAmsMathEnvironmentsForLineno{alignat}
\patchBothAmsMathEnvironmentsForLineno{gather}
\patchBothAmsMathEnvironmentsForLineno{multline}
}

\renewcommand{\thefootnote}{} 

\allowdisplaybreaks

\theoremstyle{plain} 
\newtheorem{theorem}{\sc Theorem}[section]
\newtheorem{lemma}[theorem]{\sc Lemma}

\newtheorem{proposition}[theorem]{\sc Proposition}
\newtheorem{claim}[theorem]{\sc Claim}

\theoremstyle{definition} 
\newtheorem{definition}[theorem]{\sc Definition}
\newtheorem{remark}[theorem]{\sc Remark}

\newtheorem*{AIuse}{\sc Declaration of AI use}
\makeatletter
  
  \@addtoreset{equation}{section}
\makeatother

\allowdisplaybreaks[3] 

\newcommand\cinput[2]{\lower#1pt\hbox{\input{#2}}}

\begin{document}

\title{On infinite knot families with the same HOMFLYPT polynomial}
\author{Yuki Matsushima}
\date{}
\renewcommand{\thefootnote}{\arabic{footnote}}  
\setcounter{footnote}{0} 

\maketitle

\begin{abstract}
In this paper, we focus on the examples constructed by Kanenobu \cite{K86_2, K25} and provide new family knots with the same HOMFLYPT polynomial invariants, which are distinguished by the Kauffman polynomial.
\end{abstract}


\section{Introduction} \label{Introduction}
The central problem in knot theory is determining whether two given knots are equivalent.
In studying this problem, polynomial invariants—starting with the Alexander polynomial and the Jones polynomial—have been developed as classical and effective tools.
Among them, the HOMFLYPT polynomial \cite{HOMFLY85, PT88} is a particularly powerful invariant.
However, many examples are known where distinct knots share the same HOMFLYPT polynomial.
In this paper, we focus on the examples constructed by Kanenobu \cite{K86_2, K25} and present their generalization.
By using these generalizations, we can construct new infinite families of knots sharing the same HOMFLYPT polynomials.
This construction may potentially offer clues toward building knots whose polynomials match those of specific knots, such as the unknot, as well as refining classical polynomial invariants.

The paper is organized as follows. 
In section \ref{Preliminaries}, we introduce some notation and recall previous results.
In section \ref{SameHOMFLYPTpoly}, we introduce the knot or link family, which shares the same HOMFLYPT polynomial (Proposition \ref{sameHOMFLYPTpoly}), and calculate its polynomial.
In section \ref{special example}, 
we present an explicit example of distinct knots satisfying Proposition \ref{sameHOMFLYPTpoly} (Theorem \ref{specialeq}) and provide its proof.

\begin{AIuse}
	AI tools were used to assist with English translation.
\end{AIuse}

\section{Preliminaries} \label{Preliminaries}
In this section, we introduce some notation and recall previous results.
First, let us recall the definition of the HOMFLYPT polynomial \cite{HOMFLY85, PT88}.
\begin{definition}
For an oriented link $L$ in $S^3$, the HOMFLYPT polynomial $P_{L}(v,z)$ of $L$ is uniquely defined by the following conditions;
\begin{enumerate}
\setlength{\leftskip}{1cm}
	\item[(1)] $P_{O}(v,z)=1$, where $O$ is an unknot.

	\item[(2)] If $L_{+},L_{-}$ and $L_{0}$ are the diagram depicted in Figure \ref{skein}, then
						\begin{align}
							v^{-1}P_{L_{+}}(v,z)-vP_{L_{-}}(v,z)=zP_{L_{0}}(v,z). \label{skeinrel}
						\end{align}
\vspace{-20pt}
	\begin{figure}[H]
		\centering
		\includegraphics[width=0.5\textwidth]{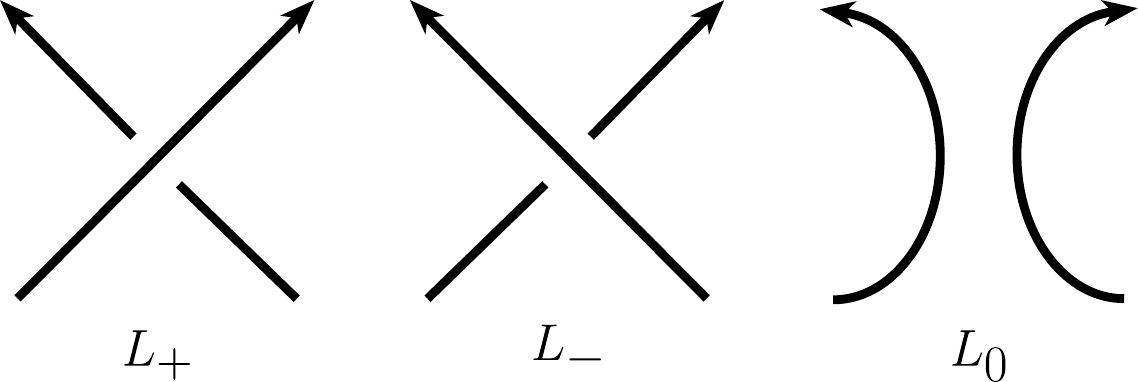}
		\caption{} \label{skein}
	\end{figure}
\end{enumerate}
\end{definition}

\begin{remark}
Let $\# L$ be the number of components of a link $L$.
Then, by \cite{LM87}*{Proposition 22}, there exists $P_{2i}(v;L) \in {\mathbb{Z}[v^{\pm1}]}$, where $i$ is a non-negative integer, such that
$$
\displaystyle{P_{L}(v,z)=\sum_{i \geq 0} P_{2i}(v;L)z^{2i+1-\#L}}.
$$
\end{remark}

There are several studies of infinite knot families with the same HOMFLYPT polynomial \citelist{\cite{K86} \cite{K89} \cite{K86_2} \cite{K25} \cite{K10} \cite{T20} \cite{W06}}.
As an example of such an infinite family of knots, we focus on the infinite family of knots given by Kanenobu \cite{K86_2, K25}.

\begin{definition}
For any $n \in \mathbb{Z} \cup \{ \infty \}$, let \includegraphics[height=1.8em, align=c]{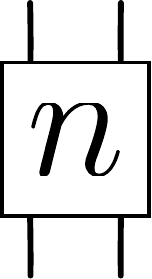} be the $n$-tangle depicted in Figure \ref{halftwist}.
	\begin{figure}[H]
		\centering
		\includegraphics[width=0.4\textwidth]{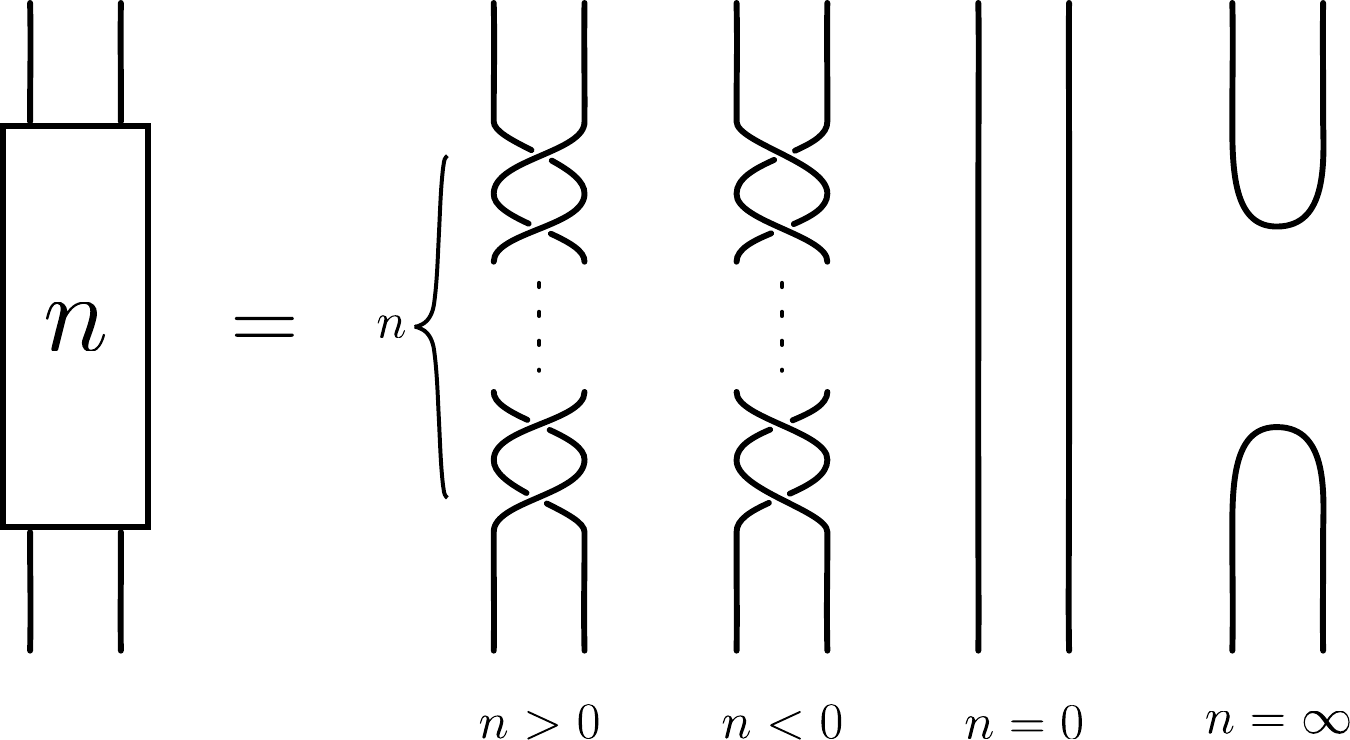}
		\caption{$n$-tangle} \label{halftwist}
	\end{figure}
\end{definition}

\begin{theorem}[\cite{K86_2}] \label{kanenobu86}
For any $p,q \in \mathbb{Z}$, let $A(p,q)$ be the knot represented by the diagram shown in Figure \ref{Kanenobuknot_86}.
	\begin{figure}[H]
		\centering
		\includegraphics[width=0.25\textwidth]{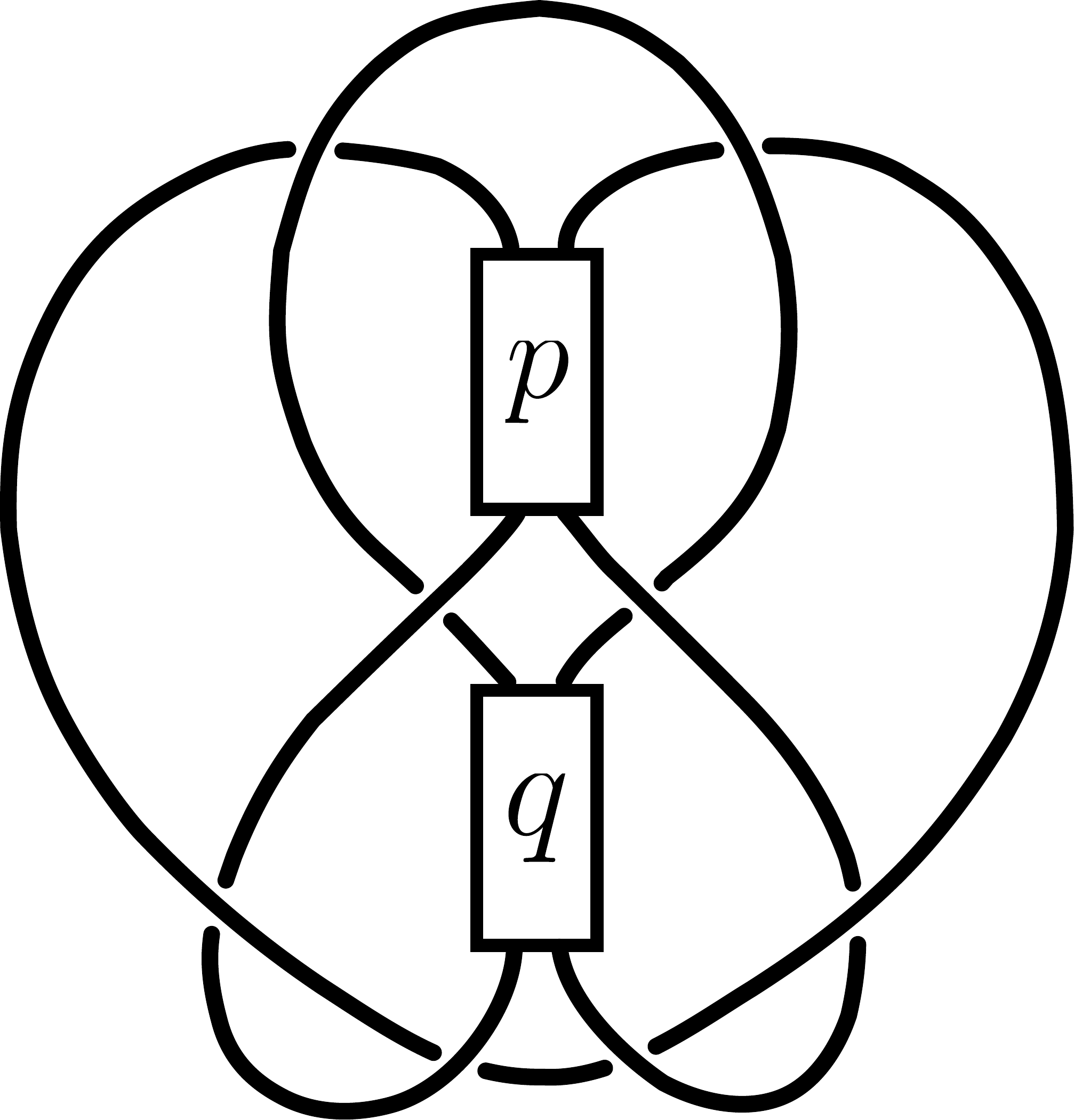}
		\caption{$A(p,q)$} \label{Kanenobuknot_86}
	\end{figure}
If $p+q = p'+q'$ and $(p, q) \equiv (p', q') \pmod{2}$, then $P_{A(p,q)}(v,z) = P_{A(p',q')}(v,z)$.
Furthermore, if $(p, q) \neq (p', q')$ and $(p, q) \neq (q', p')$, then $A(p, q) \not\cong A(p', q')$.
\end{theorem}

\begin{remark}[\cite{K89}] \label{kanenobu89}
For any element $\beta$ of the 3-braid group $B_3$, let $A_{\beta}(p,q)$ be the knot represented by the diagram shown in Figure \ref{Kanenobuknot_89}.
	\begin{figure}[H]
		\centering
		\includegraphics[width=0.185\textwidth]{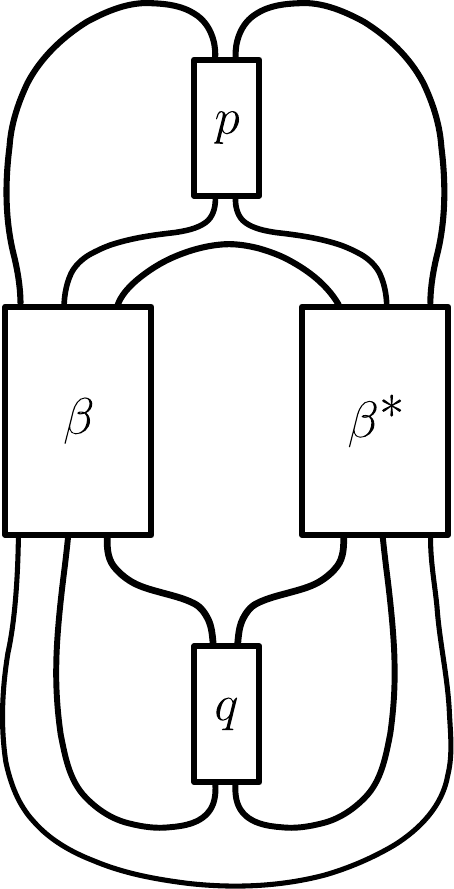}
		\caption{$A_{\beta}(p,q)$} \label{Kanenobuknot_89}
	\end{figure}
Here, $\beta^* \in B_3$ is defined by replacing the generators $\sigma_1^{\pm1}$ and $\sigma_2^{\pm1}$ in $\beta$ with $\sigma_2^{\mp1}$ and $\sigma_1^{\mp1}$, respectively.
This infinite family of knots is also known to share the same HOMFLYPT polynomial.
In particular, if we let $\beta = \sigma_2^2 \sigma_1^{-2}$, then $A(p, q)$ represents Figure \ref{Kanenobuknot_86}.
\end{remark}

\begin{theorem}[\cite{K25}] \label{kanenobu25}
For any $p \in \mathbb{Z}$, let $B(p)$ be the knot represented by the diagram shown in Figure \ref{Kanenobuknot_25_1p}.
\begin{figure}[H]
	\centering
	\includegraphics[width=0.2\textwidth]{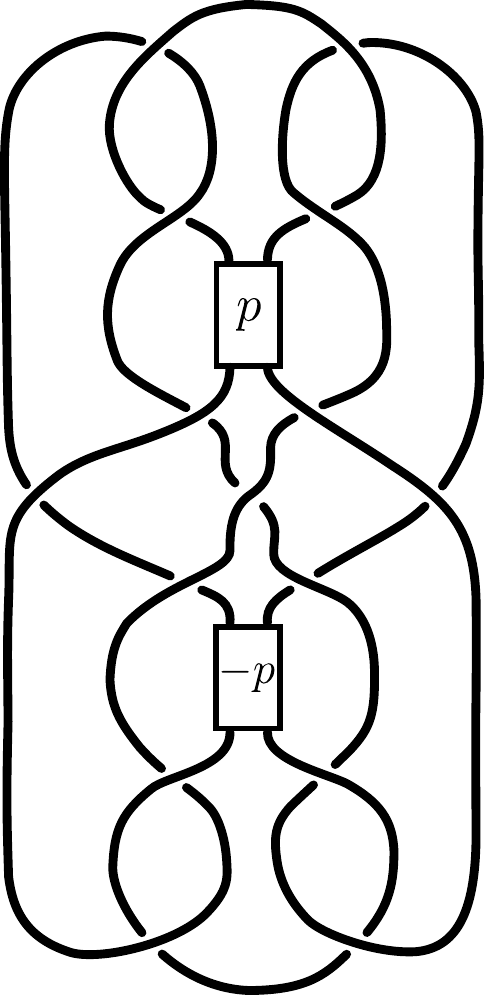}
	\caption{$B(p)$} \label{Kanenobuknot_25_1p}
\end{figure}
If $p \equiv p' \pmod{2}$, then $P_{B(p)}(v,z) = P_{B(p')}(v,z)$.
Furthermore, if $p \not= \pm p'$, then $B(p) \not\cong B(p')$.
\end{theorem}

These infinitely many knots are distinguished by the Kauffman polynomial \cite{Kau87, Kau90}.

\begin{definition}
The writhe $w(D)$ of a diagram $D$ of an oriented link is the sum of
the signs of the crossings of $D$, where each crossing has sign $+ 1$ or $-1$ as defined in Figure \ref{sign}.
	\begin{figure}[H]
	\centering
	\includegraphics[width=0.45\textwidth]{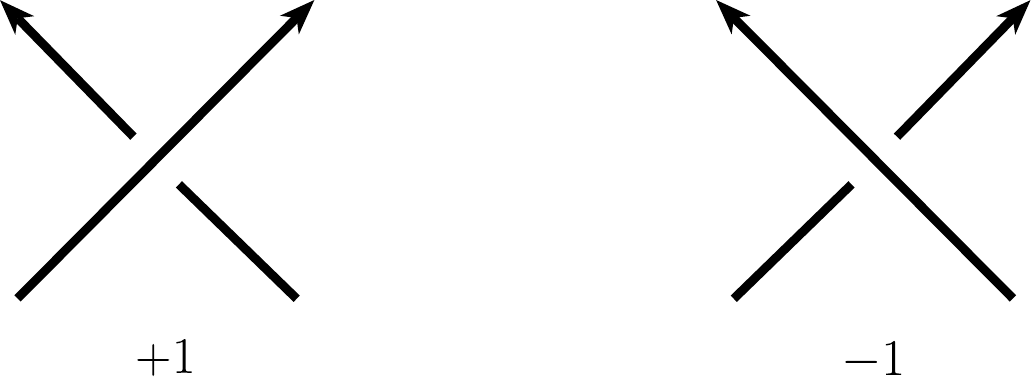}
	\caption{} \label{sign}
	\end{figure}
\end{definition}

\begin{definition}
For a diagram $D$ of an oriented link $L$, there uniquely exists $\Lambda_{D}(a,z) \in \mathbb{Z}[a^{\pm1}, z^{\pm1}]$ satisfying the following:
	\begin{enumerate}
		\item $\Lambda_U(a,z)=1$, where $U$ is the zero-crossing diagram of the unknot.
		\item $\Lambda_D(a,z)$ is unchanged by Reidemeister moves of Types $I\hspace{-1.2pt}I$ and $I\hspace{-1.2pt}I\hspace{-1.2pt}I$ on the diagram.
		\item $
\Lambda_{{\includegraphics[height=0.75em,align=c]{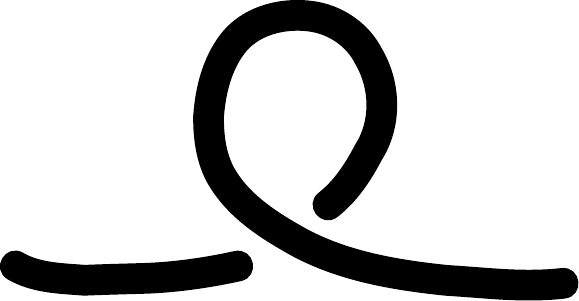}}}(a,z)=a\Lambda_{{\includegraphics[height=0.1em]{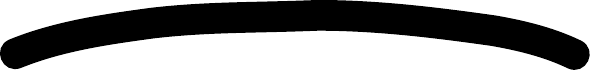}}}(a,z),\quad\Lambda_{{\includegraphics[height=0.75em,align=c]{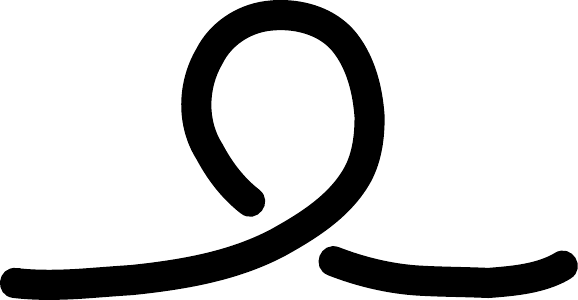}}}(a,z)=a^{-1}\Lambda_{{\includegraphics[height=0.1em]{kauffman_brakcet_R1_smooth.pdf}}}(a,z)$.
		\item If $D_{+},D_{-},D_{0},D_{\infty}$ are the diagram depicted in Figure \ref{skein_kauffman}, then
				\begin{align}
					\Lambda_{D_{+}}(a,z)+\Lambda_{D_{-}}(a,z)=z(\Lambda_{D_{0}}(a,z)+\Lambda_{D_{\infty}}(a,z)). \label{skein_kauffmanrel}
				\end{align}
	\end{enumerate}
\vspace{-10pt}
	\begin{figure}[H]
					\centering
					\includegraphics[width=0.7\textwidth]{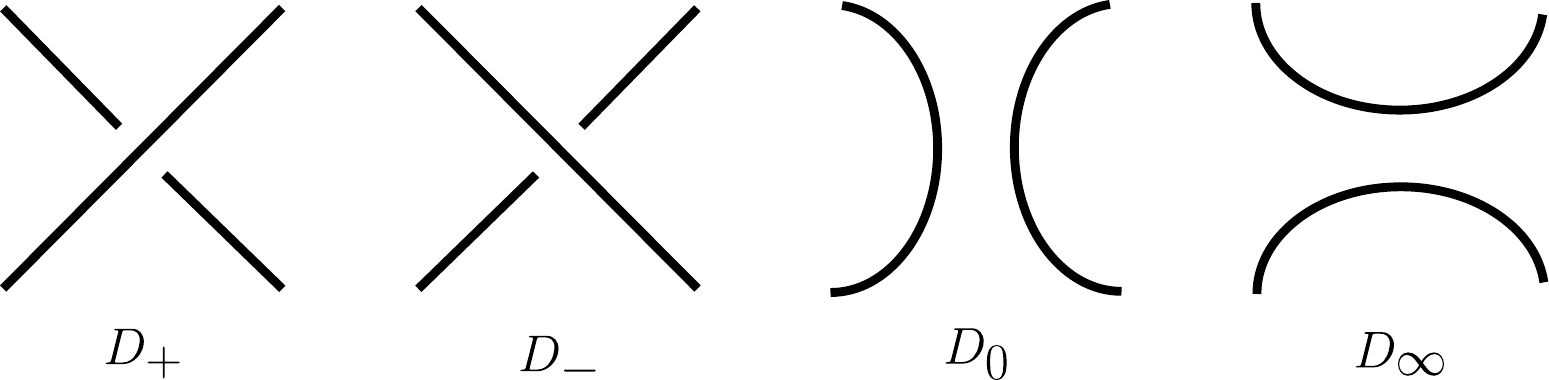}
					\caption{} \label{skein_kauffman}
		\end{figure}
\vspace{-15pt}
The Kauffman polynomial $F_L(a,z) \in \mathbb{Z}[a^{\pm1}, z^{\pm1}]$ is defined by $F_L(a,z)=a^{-w(D)}\Lambda_{D}(a,z) \in \mathbb{Z}[a^{\pm1},z^{\pm1}]$.
\end{definition}

\begin{remark}
Similarly to the HOMFLYPT polynomial, there exists $F_{i}(a;L) \in {\mathbb{Z}[a^{\pm1}]}$, where $i$ is a non-negative integer, such that
$$
\displaystyle{F_L(a,z)=\sum_{i \geq 0} F_{i}(a;L)z^{i+1-\#L}}.
$$
\end{remark}

\section{On the HOMFLYPT polynomial} \label{SameHOMFLYPTpoly}


In what follows, for any $k \in \mathbb{N}$, we choose elements $\beta_1, \dots, \beta_k \in B_3$ such that the diagram in Figure \ref{beta_knot} represents a knot.
	\begin{figure}[H]
		\centering
		\includegraphics[width=0.25\textwidth]{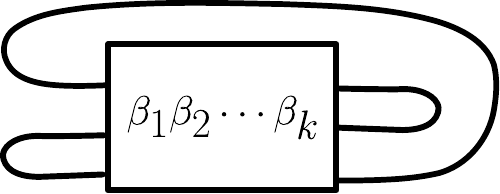}
		\caption{} \label{beta_knot}
	\end{figure}
	
\begin{proposition} \label{sameHOMFLYPTpoly}
For any $n_0, n_1, \dots, n_k \in \mathbb{Z}$, let $M_{\beta_1, \beta_2, \cdots, \beta_k}(n_0, n_1 ,\cdots, n_k)$ be the knot represented by the diagram shown in Figure \ref{Mknot_2607}.
	\begin{figure}[H]
		\centering
		\includegraphics[width=0.85\textwidth]{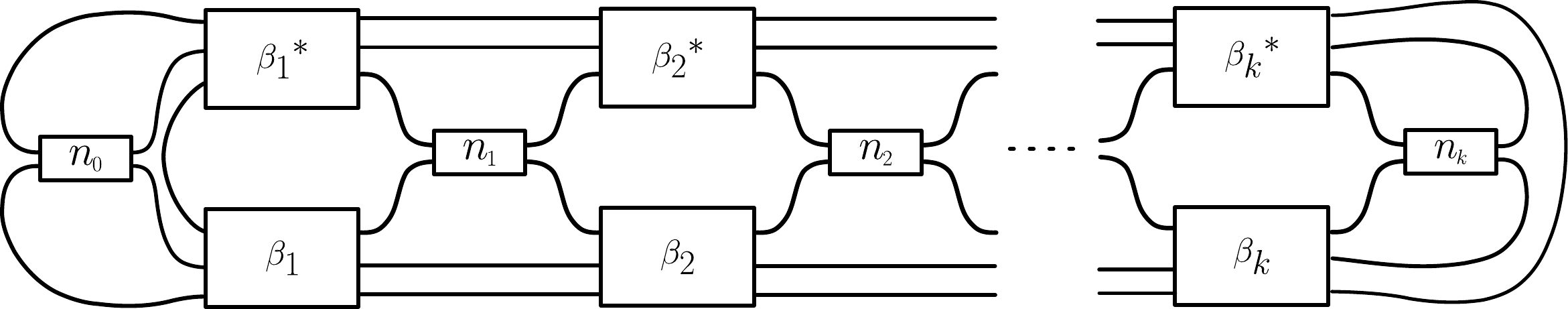}
		\caption{$M_{\beta_1, \beta_2, \cdots, \beta_k}(n_0, n_1 ,\cdots, n_k)$} \label{Mknot_2607}
	\end{figure}
Then, the HOMFLYPT polynomial of $M_{\beta_1, \beta_2, \cdots, \beta_k}(n_0, n_1 ,\cdots, n_k)$ is given by
\begin{align*}
		&P_{M_{\beta_1, \beta_2, \cdots, \beta_k}(n_0, n_1 ,\cdots, n_k)}(v,z) \\	
  		&= \begin{cases}
    		 v^{-n_0-n_k}(P_{M_{\beta_1, \beta_2, \cdots, \beta_k}(0, n_1 ,\cdots, n_{k-1} , 0)}(v,z)-1)+1, & \text{if $(n_0,n_k) \equiv (0,0) \pmod 2$,} \\
    		 v^{1-n_0-n_k}(P_{M_{\beta_1, \beta_2, \cdots, \beta_k}(1, n_1 ,\cdots, n_{k-1} , 0)}(v,z)-1)+1, & \text{if $(n_0,n_k) \equiv (1,0) \pmod 2$,} \\
    		 v^{1-n_0-n_k}(P_{M_{\beta_1, \beta_2, \cdots, \beta_k}(0, n_1 ,\cdots, n_{k-1} , 1)}(v,z)-1)+1, & \text{if $(n_0,n_k) \equiv (0,1) \pmod 2$,} \\
    		 v^{-n_0-n_k}(P_{M_{\beta_1, \beta_2, \cdots, \beta_k}(1, n_1 ,\cdots, n_{k-1} , -1)}(v,z)-1)+1, & \text{if $(n_0,n_k) \equiv (1,1) \pmod 2$.}
  		\end{cases}
	\end{align*}
\end{proposition}

We use the following lemma to prove Proposition \ref{sameHOMFLYPTpoly}.

\begin{lemma} \label{triviallink}
For any $n_0, n_1, \dots, n_k \in \mathbb{Z} \cup \{ \infty \}$, we have $M_{\beta_1, \beta_2, \cdots, \beta_k}(\infty, n_1 ,\cdots, n_k) \cong M_{\beta_1, \beta_2, \cdots, \beta_k}(n_0, n_1 ,\cdots, \infty) \cong U^2$, where $U^2$ is the trivial 2-component link.
\end{lemma}

\begin{proof}
When $n_0 = \infty$, since the orientation inside the tangle is anti-parallel, the diagram can be deformed as shown in Figure \ref{Mknot_2607_deformation}.
	\begin{figure}[H]
		\centering
		\includegraphics[width=1.0\textwidth]{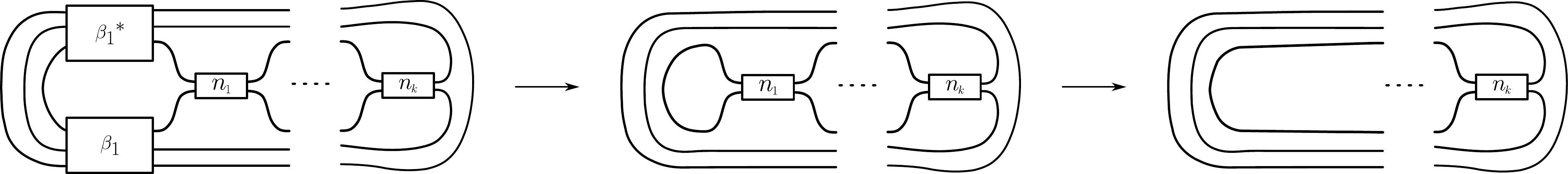}
		\caption{$M_{\beta_1, \beta_2, \cdots, \beta_k}(\infty, n_1 ,\cdots, n_k)$} \label{Mknot_2607_deformation}
	\end{figure}
This is the desired conclusion.
The case $n_k = \infty$ can be shown similarly, and we can prove Lemma \ref{triviallink}.
\end{proof}

\begin{proof}[\indent\sc{Proof of Proposition \ref{sameHOMFLYPTpoly}.}]
We consider crossing changes and smoothings at the crossings in the $n_0$- and $n_k$-tangles.
By the symmetry of the diagram, note that the orientation of both tangles is anti-parallel.
The knots obtained by a crossing change and a smoothing at a crossing in the $n_0$-tangle  are $M_{\beta_1, \beta_2, \cdots, \beta_k}(n_0-2, n_1,\cdots, n_k)$, and are $U^2$ by Lemma \ref{triviallink}, respectively.
Similarly, the knots obtained by a crossing change and a smoothing at a crossing in the $n_k$-tangle are $M_{\beta_1, \beta_2, \cdots, \beta_k}(n_0, n_1,\cdots, n_k-2)$, and are $U^2$ by Lemma \ref{triviallink}, respectively.
For simplicity, let  $P(n_0, n_1 ,\cdots, n_{k-1} , n_k)(v,z)$ be the HOMFLYPT polynomial of $M_{\beta_1, \beta_2, \cdots, \beta_k}(n_0, n_1 ,\cdots, n_{k-1} , n_k)$.
When $n_0$ and $n_k$ are positive and even, we have
\begin{align*}
	P(n_0, n_1 ,\cdots, n_{k-1} , n_k)(v,z) &= v^{-n_0} P(0, n_1 ,\cdots, n_{k-1} , n_k)(v,z) \\
											          &\quad  - (v^{-n_0+2} + \cdots + v^{-2} + 1)v^{-1}z P_{U^2}(v,z) \\
											          &= v^{-n_0-n_k} P(0, n_1 ,\cdots, n_{k-1} , 0)(v,z) \\
											          &\quad  - (v^{-n_0-n_k+2} + \cdots + v^{-2} + 1)v^{-1}z P_{U^2}(v,z)
\end{align*}
by the recursive formula for $P(n_0, n_1 ,\cdots, n_{k-1} , n_k)$.
Since $P_{U^2}(v,z) = (v^{-1}-v)z^{-1}$, we have 
\begin{align*}
	P(n_0, n_1 ,\cdots, n_{k-1} , n_k)(v,z) &= v^{-n_0-n_k} P(0, n_1 ,\cdots, n_{k-1} , 0)(v,z) \\
											          &\quad  - (v^{-n_0-n_k+2} + \cdots + v^{-2} + 1)v^{-1}z P_{U^2}(v,z) \\
											          &= v^{-n_0-n_k} (P(0, n_1 ,\cdots, n_{k-1} , 0)(v,z) -1) + 1.
\end{align*}
In the case where either $n_0$ or $n_k$ is a negative even integer, a similar calculation leads to the same conclusion.

Furthermore, when at least one of $n_0$ and $n_k$ is odd, by using
\begin{align*}
	v^{-1} P(1, n_1 ,\cdots, n_{k-1} , n_k)(v,z) -v P(-1, n_1 ,\cdots, n_{k-1} , n_k)(v,z) &= zP_{U^2}(v,z), \\
	v^{-1} P(n_0, n_1 ,\cdots, n_{k-1} , 1)(v,z) -v P(n_0, n_1 ,\cdots, n_{k-1} , -1)(v,z) &= zP_{U^2}(v,z)
\end{align*}
and we obtain the desired formula similarly.
This completes the proof.
\end{proof}


\section{special example} \label{special example}
\subsection{Special example} \label{prop sp. e.g.}
For fixed $\beta_1, \beta_2, \cdots, \beta_k \in B_3$, $\{ M_{\beta_1, \beta_2, \cdots, \beta_k}(n_0, n_1 ,\cdots, n_k) \mid n_0, n_1 ,\cdots, n_k \in \mathbb{Z} \}$ are infinitely many knots sharing the same HOMFLYPT polynomial, but some of them maybe same.
In this section, in some cases, we show that they are indeed distinct knots.
In what follows, we consider the case $\beta_1= \sigma_2^2, \beta_2 = \cdots =\beta_{3i} = \sigma_1^{-1} \sigma_2, \beta_{3i+1} = \sigma_1^{-1} \sigma_2 \sigma_1^{-1} $ and $n_1 = n_2 = \cdots = n_{3i} = 2$ in Figure \ref{Mknot_2607}.

\begin{definition}
For any $p, q \in \mathbb{Z} \cup \{\infty\}$ and $i \in \mathbb{N} \cup \{ 0 \}$, let $K_{i}^0(p, q)$ be the link represented by the diagram shown in Figure \ref{example_2607_(n_1 n_2)_0}.
	\begin{figure}[H]
		\centering
		\includegraphics[width=0.65\textwidth]{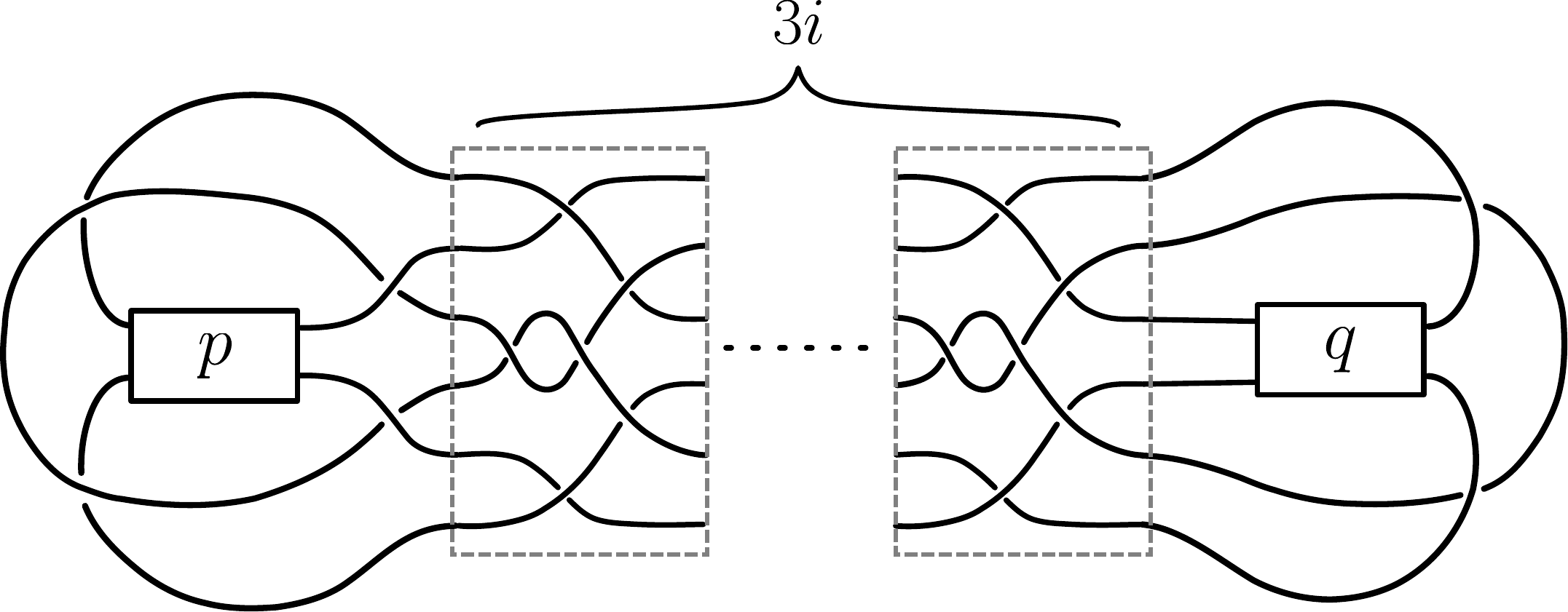}
		\caption{$K_{i}^0(p, q)$} \label{example_2607_(n_1 n_2)_0}
	\end{figure}
\end{definition}

\begin{theorem}\label{specialeq}
If $n, n'\in \mathbb{N}$ and $n \equiv n' \pmod{2}$, then $P_{K_{i}^0(n,-n)}(v,z) = P_{K_{i}^0(n',-n')}(v,z)$ for any $i \in \mathbb{N}$.
Furthermore, if $n \not= \pm n'$, then $K_{i}^0(n,-n) \not\cong K_{i}^0(n',-n')$.
\end{theorem}

\begin{proof}[\indent\sc{Proof of Proposition \ref{specialeq}.}]
From Proposition \ref{sameHOMFLYPTpoly}, we have
	\begin{equation}
	  P_{K_{i}^0(n,-n)}(v,z) =
	  \begin{cases*}
	    P_{K_{i}^0(0,0)}(v,z) 	& $n \equiv 0 \pmod{2}$, \\
	    P_{K_{i}^0(1,-1)}(v,z) 	& $n \equiv 1 \pmod{2}$.
	  \end{cases*}
	\end{equation}
Thus, if $n \equiv n' \pmod{2}$, then $P_{K_{i}^0(n,-n)}(v,z) = P_{K_{i}^0(n',-n')}(v,z)$.
To show that these knots are distinct, we calculate the 2nd coefficient of the Kauffman polynomial of them.

Let $F_m(p, q)$ be the $m$-th coefficient polynomial of the Kauffman polynomial of $K_{i}^0(p, q)$.
Under this setting, we calculate $F_2(n, -n)$.

If $p, q \in \mathbb{Z}$, the writhe of the diagram Figure \ref{example_2607_(n_1 n_2)_0} is $-(p+q+6i)$.
If exactly one of $p$ and $q$ is $\infty$, the writhe of diagram Figure \ref{example_2607_(n_1 n_2)_0} obtained by applying Reidemeister move I to the tangle is $-(p+q+6i-1)$.

By Lemma \ref{triviallink}, $K_i^0(n,\infty) \cong K_i^0(\infty,n) \cong U^2$, we have $F_m (n,\infty) = F_m(\infty,n) = F_m(U^2)$, where $F_m(U^2)$ is the $m$-th coefficient polynomial of the Kauffman polynomial for $U^2$.

When $n$ is even, by combining these with the recursive formula for the Kauffman polynomial \cite{K89}*{Proposition 2.2} and comparing each coefficient polynomial of the Kauffman polynomial, we have
\begin{align*}
						F_2(n,-n)	&= a^{n} \big( (-1)^{\frac{n}{2}} F_2(0,-n) - (-1)^{\frac{n}{2}} \frac{n}{2} a^{-1}  F_1(1,-n) + B_n F_1(U^2) \\
										&\quad - (-1)^{\frac{n}{2}} \frac{n^2-2n}{8}F_0(0,-n) + C_n  F_0(U^2) \big)\\
										&= F_2(0,0) - \frac{n}{2} a F_1(0,-1) + (-1)^{\frac{n}{2}} B_{-n} F_1(U^2) \\
										&\quad - \frac{n^2 - 2n}{8}F_0(0,0) + (-1)^{\frac{n}{2}} C_{-n} F_0(U^2) \\
										&\quad - \frac{n}{2} a^{-1} \bigg( F_1(1,0) + (-1)^{\frac{n}{2}} A_{-n} F_1(U^2) - \frac{n}{2} a F_0(1,-1) + (-1)^{\frac{n}{2}} B_{-n} F_0(U^2)  \bigg) \\
  										&\quad - \frac{n^2-2n}{8} \big( F_0(0,0) + (-1)^{\frac{n}{2}} A_{-n} F_0(U^2) \big) \\
  										&\quad + a^{n}B_n F_1(U^2) + a^{n}C_n F_0(U^2), 
\end{align*}

where

\begin{equation*}
	\sigma_n =
  \begin{cases}
    1 & \text{if $n \geq 0$,} \\
    -1 & \text{if $n < 0$,}
  \end{cases}
\end{equation*}

and $A_n, B_n, C_n$ are defined by  

\begin{equation*}
	A_n =
  \begin{cases}
    \displaystyle{ \sum_{i=1}^{\frac{|n|}{2}} (-1)^{i+1} a^{\sigma_n(2i-|n|-1)} } & \text{if $n \equiv 0 \pmod{2}$,} \\
    \displaystyle{ \sum_{i=1}^{\frac{|n|-1}{2}} (-1)^{i+1} a^{\sigma_n(2i-|n|-1)} } & \text{if $n \equiv 1 \pmod{2}$,}
  \end{cases}
\end{equation*}

\begin{equation*}
	B_n =
  \begin{cases}
    \displaystyle{\sum_{i=1}^{\frac{|n|}{2}-1} (-1)^{i+1} i a^{\sigma_n(2i - |n| )} } & \text{if $n \equiv 0 \pmod{2}$,} \\
    \displaystyle{ \sum_{i=1}^{\frac{|n|-1}{2}} (-1)^{i+1} i a^{\sigma_n(2i - |n|)} } & \text{if $n \equiv 1 \pmod{2}$,}
  \end{cases}
\end{equation*}

\begin{equation*}
	C_n =
  \begin{cases}
    \displaystyle{ \sum_{i=1}^{\frac{|n|}{2}-1} (-1)^{i+1} \frac{i(i+1)}{2} a^{\sigma_n(2i - |n| + 1)} } & \text{if $n \equiv 0 \pmod{2}$,} \\
    \displaystyle{ \sum_{i=1}^{\frac{|n|-1}{2}-1} (-1)^{i+1} \frac{i(i+1)}{2} a^{\sigma_n(2i - |n| + 1)} } & \text{if $n \equiv 1 \pmod{2}$.}
  \end{cases}
\end{equation*}

Simplifying this, we have
\begin{align*}  						
						F_2(n,-n) &= F_2(0,0) - \frac{n^2-2n}{4} F_0(0,0) + \frac{n^2}{4} F_0(1,-1) \\
									 & - \frac{n}{2} \big(a F_1(0,-1) + a^{-1} F_1(1,0) - F_1(U^2) \big),			
\end{align*}  
where we use the identities
\begin{align*}
	&(-1)^{\frac{n}{2}}B_{-n} - (-1)^{\frac{n}{2}} \frac{n}{2} a^{-1}A_{-n} + a^{n}B_n = \frac{n}{2}, \\
	&(-1)^{\frac{n}{2}}C_{-n} - (-1)^{\frac{n}{2}} \frac{n}{2} a^{-1}B_{-n} - (-1)^{\frac{n}{2}} \frac{n^2-2n}{8}A_{-n} + a^{n}C_n = 0.
\end{align*}

Note that the diagram of $K_i^0(q, p)$ is obtained from the diagram of $K_i^0(p, q)$ by rotating along the $y$-axis by the angle  $180^\circ$. 
Since $K_i^0(0,1) \cong K_i^0(1,0)$ and
\begin{align*}
	F_1(0,-1) = -a^{-2} F_1(0,1) + a^{-1}\big( F_1(0,\infty) +F_0(0,0) \big), 
\end{align*}
and we obtain
\begin{align*}
	F_2(n,-n)	- F_2(0,0) = -\frac{n^2}{4} \big( F_0(0,0) - F_0(1,-1) \big).	 
\end{align*}

Similarly, when $n$ is odd, we obtain
\begin{align*}
  F_2(n,-n)	- F_2(1,-1) = \frac{n^2-1}{4} \big( F_0(0,0) - F_0(1,-1) \big).	
\end{align*}
Now, we use the following proposition, which is proved in the next section.

\begin{proposition} \label{K_{i}^j(n,-n)}
For any $p, q \in \mathbb{Z} \cup \{\infty\}$ and $i \in \mathbb{N}$, let $K_{i}^1(p, q)$ be the link represented by the diagram shown in Figure \ref{example_2607_2_(n_1 n_2)_1}.
	\begin{figure}[H]
		\centering
		\includegraphics[width=0.65\textwidth]{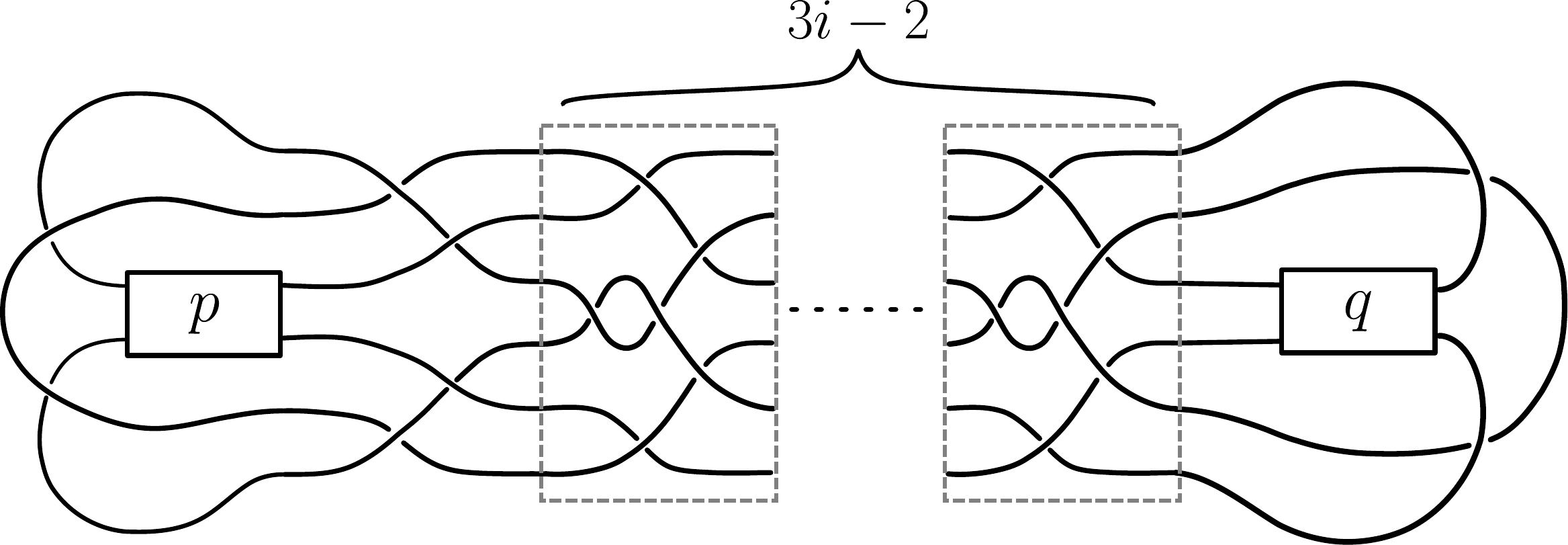}
		\caption{$K_{i}^1(p, q)$} \label{example_2607_2_(n_1 n_2)_1}
	\end{figure}
Then, for $j=0,1$ and any $i \in \mathbb{N}$, we have $F_0(a; K_i^j(1,-1)) \not= F_0(a; K_i^j(0,0))$.
\end{proposition}

By Proposition \ref{K_{i}^j(n,-n)}, we obtain
\begin{align*}
	F_2(n,-n)	 - F_2(0,0) = -\frac{n^2}{4} \big( F_0(0,0) - F_0(1,-1) \big) \not= 0.	 
\end{align*}
when $n$ is even, and
\begin{align*}
	F_2(n,-n) - F_2(1,-1) = \frac{n^2-1}{4} \big( F_0(0,0) - F_0(1,-1) \big) \not= 0.	
\end{align*}
when $n$ is odd.
From this, $F_2(n,-n) = F_2(n',-n')$ holds if and only if $n= \pm n'$.
Therefore, if $n \not= \pm n'$, then $K_{i}^0(n,-n) \not\cong K_{i}^0(n',-n')$.
This completes the proof of Theorem~\ref{specialeq}.
\end{proof}

\subsection{The Proof of Propositon \ref{K_{i}^j(n,-n)}} 

In this section, we prove Propositon \ref{K_{i}^j(n,-n)}.
For any link $L$, since $F_0(a; L) = P_0(\sqrt{-1}v; L)$, it suffices to show that $P_0(v; K_i^j(1,-1)) \not= P_0(v; K_i^j(0,0))$.
For simplicity, we denote $P_0(\sqrt{-1}v; L)$ by $P_0(L)$.
We consider a crossing change and a smoothing at the two crossings in Figure \ref{example_2607_(n_1 n_2)}, $c_1, c_2 $, of $K^j_i(p, q)$.
	\begin{figure}[H]
		\centering
		\includegraphics[width=0.65\textwidth]{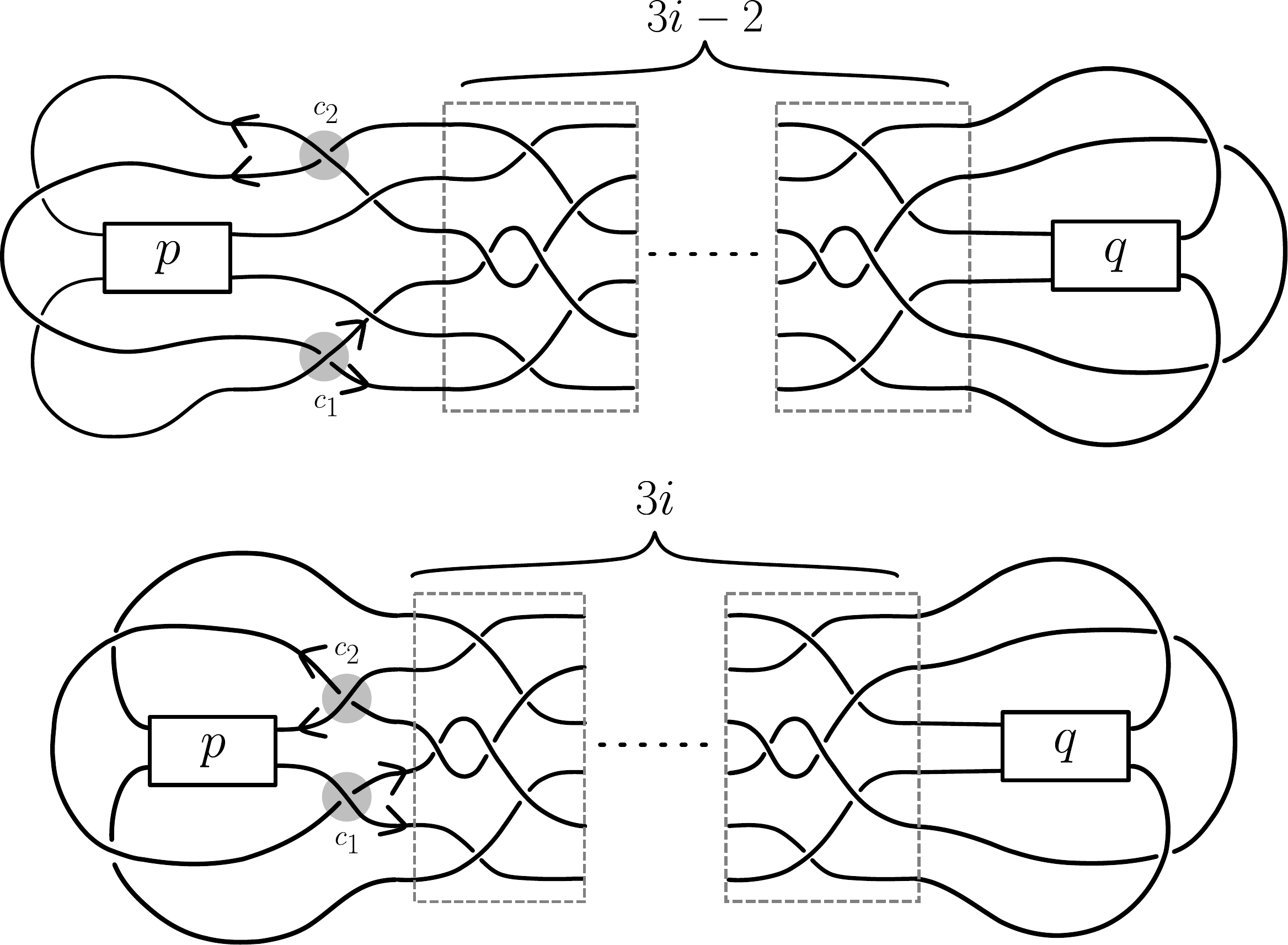}
		\caption{$K_{i}^1(p, q)$ and $K_{i}^0(p, q)$} \label{example_2607_(n_1 n_2)}
	\end{figure}
Here, the orientations are given as shown in Figure \ref{example_2607_(n_1 n_2)}.
For simplicity, we denote $K_{i}^j(n,-n)$ by $K_{i}^j(n)$.
Furthermore, we denote by $K_i^1(n)_{+ -}$, $K_i^1(n)_{0 0}$, $K_i^1(n)_{+ 0}$, and $K_i^1(n)_{0 -}$ the knots obtained from $K_i^1(n)$ 
by changing $c_1$ and $c_2$ to a positive crossing and a negative crossing, 
by smoothing $c_1$ and $c_2$, 
by changing $c_1$ to a positive crossing and by smoothing $c_2$, 
and by smoothing $c_1$ and by changing $c_2$ to a negative crossing, respectively.
We define $K_i^0(n)_{- +}$, $K_i^0(n)_{0 0}$, $K_i^0(n)_{- 0}$, and $K_i^0(n)_{0 +}$ similarly.
Note that,
\begin{align*}
	K_{i}^1(n)_{+ -} = K_{i-1}^0(n+2,-n) ,\quad K_{i}^0(n)_{- +} = K_{i}^1(n+2,-n).
\end{align*}
Equation~\eqref{skeinrel} and a direct calculation yield
\begin{align*} 
	P_0(K_{i}^1(n)) &= P_0(K_{i}^1(n)_{+-}) + v^{-1} P_0(K_{i}^1(n)_{+0}) -v P_0(K_{i}^1(n)_{0-}) - P_0(K_{i}^1(n)_{00}) \\
						&= v^{-2} P_0(K_{i-1}^0(n)) - v^{-1} P_0(U^2) \\
					   &\quad + v^{-1} P_0(K_{i}^1(n)_{+0}) -v P_0(K_{i}^1(n)_{0-}) - P_0(K_{i}^1(n)_{00}),\\
	P_0(K_{i}^0(n)) &= P_0(K_{i}^0(n)_{-+}) - v P_0(K_{i}^0(n)_{-0}) +v^{-1} P_0(K_{i}^0(n)_{0+}) -  P_0(K_{i}^0(n)_{00}) \\
						&= v^{-2} P_0(K_{i}^1(n)) - v^{-1} P_0(U^2) \\
					   &\quad - v P_0(K_{i}^0(n)_{-0}) + v^{-1} P_0(K_{i}^0(n)_{0+}) - P_0(K_{i}^0(n)_{00}).
\end{align*}
Iterating the above two equations with respect to $i$, we obtain
\begin{align*} 
	P_0(K_{i}^1(n)) &= v^{-4i+2} P_0(K_{0}^0(n)) + \frac{(-v^{-1})(1-v^{-4i+2})}{1-v^{-2}}P_0(U^2) \\
					   &\quad -\sum^{i}_{l=1} v^{-4(i-l)} P_0(K_{l}^1(n)_{00}) -v^{-2} \sum^{i-1}_{l=1} v^{-4(i-1-l)} P_0(K_{l}^0(n)_{00})\\
					   &\quad +\sum^{i}_{l=1} v^{-4(i-l)} (v^{-1} P_0(K_{l}^1(n)_{+0}) -v P_0(K_{l}^1(n)_{0-})), \\
					   &\quad +v^{-2} \sum^{i-1}_{l=1} v^{-4(i-1-l)} (- v P_0(K_{l}^0(n)_{-0}) + v^{-1} P_0(K_{l}^0(n)_{0+})),\\
	P_0(K_{i}^0(n)) &= v^{-4i} P_0(K_{0}^0(n)) + \frac{(-v^{-1})(1-v^{-4i})}{1-v^{-2}}P_0(U^2) \\
					   &\quad - v^{-2}\sum^{i}_{l=1} v^{-4(i-l)} P_0(K_{l}^1(n)_{00}) - \sum^{i}_{l=1} v^{-4(i-l)} P_0(K_{l}^0(n)_{00})\\
					   &\quad +v^{-2} \sum^{i}_{l=1} v^{-4(i-l)} (v^{-1} P_0(K_{l}^1(n)_{+0}) -v P_0(K_{l}^1(n)_{0-})) \\
					   &\quad +\sum^{i}_{l=1} v^{-4(i-l)} (- v P_0(K_{l}^0(n)_{-0}) + v^{-1} P_0(K_{l}^0(n)_{0+})).
\end{align*}

We prove that $P_0(K_{i}^1(1)) \not= P_0(K_{i}^1(0))$ and $P_0(K_{i}^0(1)) \not= P_0(K_{i}^0(0))$.

First, we show that $P_0(K_{i}^1(1)) \not= P_0(K_{i}^1(0))$.
In the above equation, $K_{0}^0(1) \cong K_{0}^0(0) \cong 3_1 \# 3_1^*$, so $P_0(K_{0}^0(1)) = P_0(K_{0}^0(0))$. 
$K_{i}^1(0)_{00}$ and $K_{i}^1(1)_{00}$ can be represented as in the figures below, respectively.
Thus $K_{i}^1(0)_{00}$ and $K_{i}^1(1)_{00}$ are 3-component links whose components are all unknots, with linking numbers $-2i+2$.
	\begin{figure}[H]
		\centering
		\includegraphics[width=0.65\textwidth]{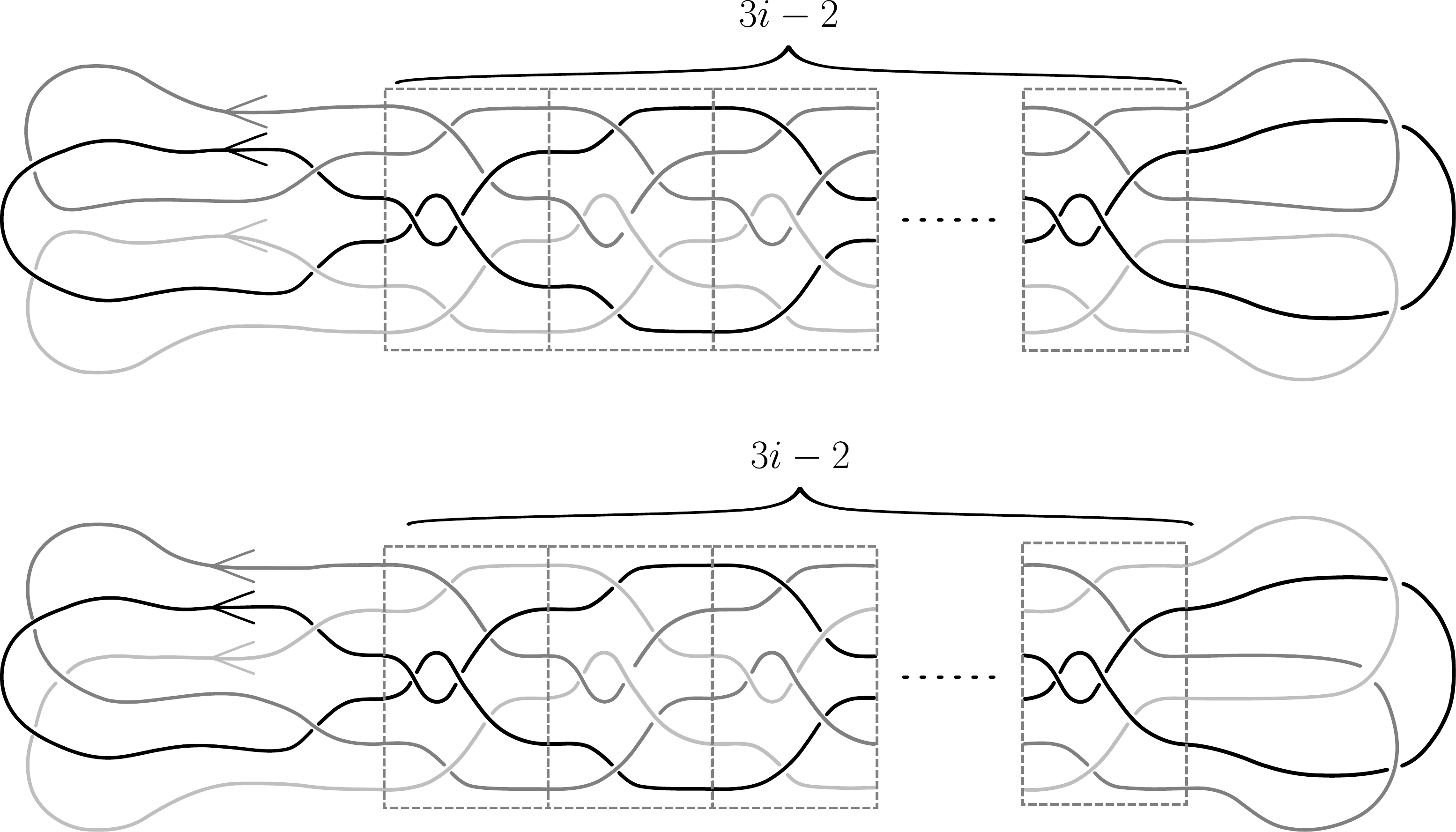}
		\caption{$K_{i}^1(0)_{00}$ and $K_{i}^1(1)_{00}$} \label{example_2607_(n_1 n_2)_00_1}
	\end{figure}

It suffices to show that
\begin{align}
	\begin{split}
			 &\sum^{i}_{l=1} v^{-4(i-l)} (v^{-1} P_0(K_{l}^1(1)_{+0}) -v P_0(K_{l}^1(1)_{0-}) \\ 
			 &+v^{-2} \sum^{i-1}_{l=1} v^{-4(i-1-l)} (- v P_0(K_{l}^0(1)_{-0}) + v^{-1} P_0(K_{l}^0(1)_{0+})) \\%
	\not= &\sum^{i}_{l=1} v^{-4(i-l)} (v^{-1} P_0(K_{l}^1(0)_{+0}) -v P_0(K_{l}^1(0)_{0-}) \\ 
			 &+v^{-2} \sum^{i-1}_{l=1} v^{-4(i-1-l)} (- v P_0(K_{l}^0(0)_{-0}) + v^{-1} P_0(K_{l}
^0(0)_{0+}))
	\end{split} \label{deg of P_0(K_i^1(n))}
\end{align}

Similarly, comparing the expressions for $P_0(K_i^0(1))$ and $P_0(K_i^0(0))$, it suffices to show that
\begin{align} 
	\begin{split}
			 &v^{-2} \sum^{i}_{l=1} v^{-4(i-l)} (v^{-1} P_0(K_{l}^1(1)_{+0}) -v P_0(K_{l}^1(1)_{0-})) \\
			 &+\sum^{i}_{l=1} v^{-4(i-l)} (- v P_0(K_{l}^0(1)_{-0}) + v^{-1} P_0(K_{l}^0(1)_{0+})) \\
	\not= &\; v^{-2} \sum^{i}_{l=1} v^{-4(i-l)} (v^{-1} P_0(K_{l}^1(0)_{+0}) -v P_0(K_{l}^1(0)_{0-})) \\
			 &+\sum^{i}_{l=1} v^{-4(i-l)} (- v P_0(K_{l}^0(0)_{-0}) + v^{-1} P_0(K_{l}^0(0)_{0+})).
	\end{split} \label{deg of P_0(K_i^0(n))}
\end{align}

Then, the following two lemmas hold.

\begin{lemma} \label{pm0}
For any $i \in \mathbb{N}$, we have $P_0(K_i^1(1)_{+ 0}) \not= P_0(K_i^1(0)_{+ 0}))$ and $P_0(K_i^0(1)_{- 0}) \not= P_0(K_i^0(0)_{- 0}))$.
\end{lemma}

\begin{proof}[Proof of Lemma \ref{pm0}]
First, we consider $P_0(K_i^1(1)_{+ 0})$ and $P_0(K_i^0(1)_{- 0})$.
The links $K_i^1(1)_{+0}$ and $K_i^0(1)_{-0}$ are 2-component links consisting of the unknot and the $2i$-twist knot, as shown in Figures \ref{example_2607_(1 -1)_+0_1} and \ref{example_2607_(1 -1)_-0_0}, respectively.
In both figures, the unknot is shown in black and the $2i$-twist knot in gray.
	\begin{figure}[H]
		\centering
		\includegraphics[width=0.65\textwidth]{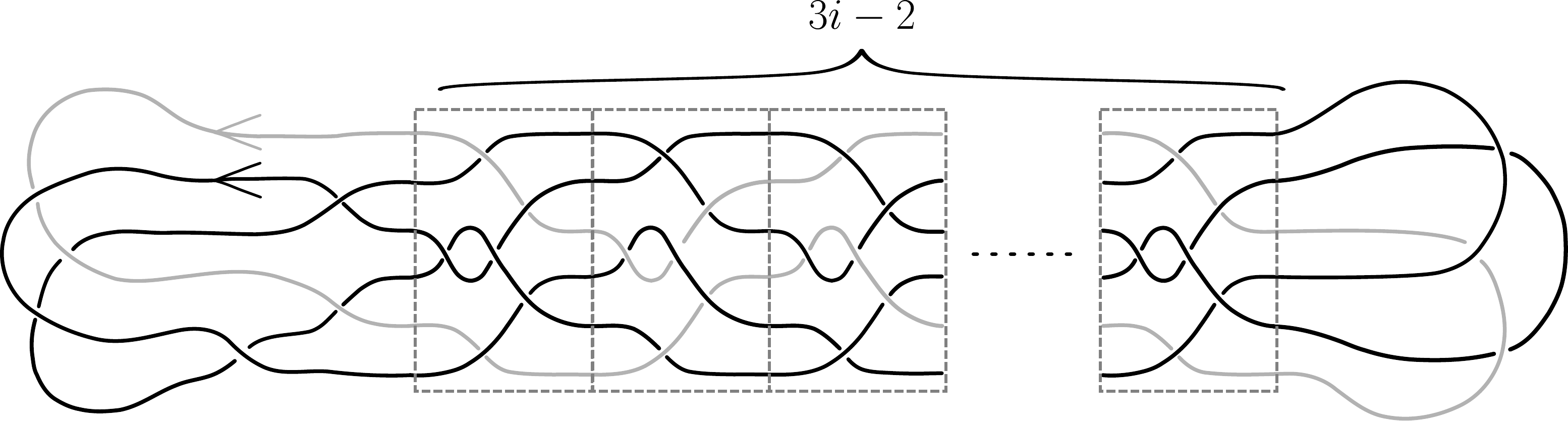}
		\caption{$K_i^1(1)_{+ 0}$} \label{example_2607_(1 -1)_+0_1}
		\vspace{1em}
		\centering
		\includegraphics[width=0.6\textwidth]{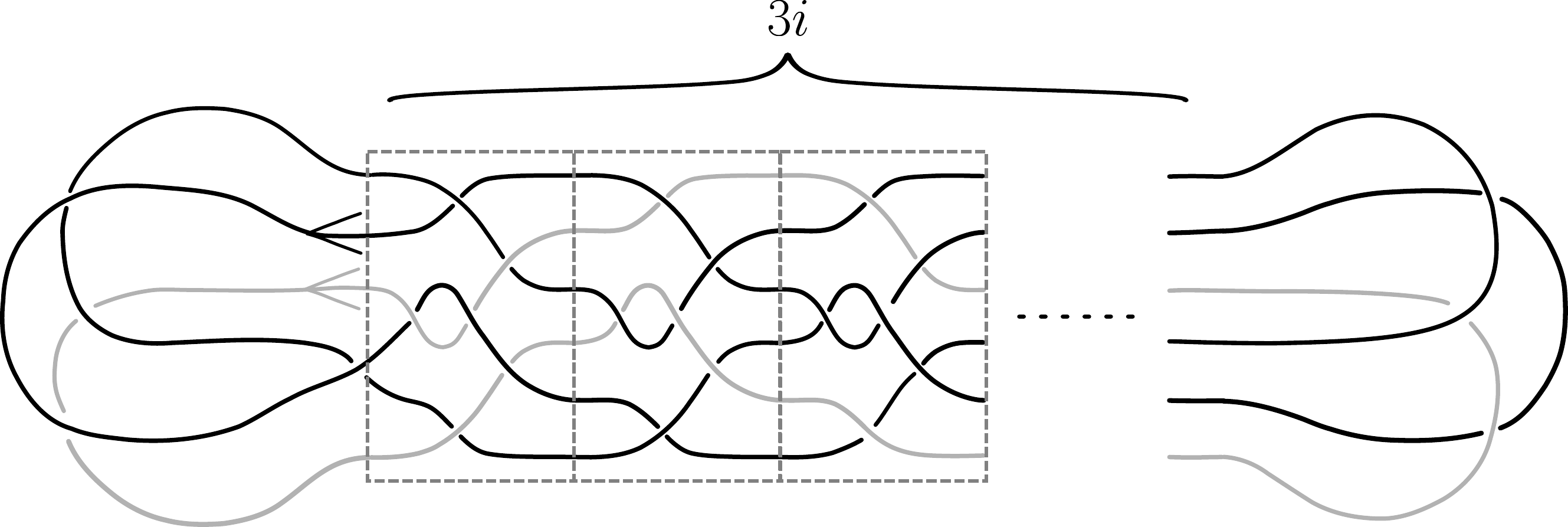}
		\caption{$K_i^0(1)_{- 0}$} \label{example_2607_(1 -1)_-0_0}
	\end{figure}
Moreover, the linking numbers of $K_i^1(1)_{+ 0}$ and $K_i^0(1)_{- 0}$ are $-2i+2$ and $-2i-1$, respectively.
Note that, by direct calculation, the $0$-th coefficient polynomial of the HOMFLY polynomial for $2i$-twist knot is $v^{-2i}-v^{-2i+2}+v^2$.
Thus, by \cite[Proposition 22]{LM87} and a direct calculation, we obtain
\begin{align}
	P_0(K_i^1(1)_{+ 0}) &= v^{-4i+4} (v^{-1}-v) (v^{-2i}-v^{-2i+2}+v^2), \label{K_i^1(1)_{+ 0}} \\
	P_0(K_i^0(1)_{- 0}) &= v^{-4i-2} (v^{-1}-v) (v^{-2i}-v^{-2i+2}+v^2). \label{K_i^0(1)_{- 0}}
\end{align}

Next, we consider $P_0(K_i^1(0)_{+ 0}))$ and $P_0(K_i^0(0)_{- 0}))$.
The links $K_i^1(0)_{+0}$ and $K_i^0(0)_{-0}$ are 2-component links, each consisting of two components isotopic to the unknot, as shown in Figures \ref{example_2607_(0 0)_+0_1} and \ref{example_2607_(0 0)_-0_0}, respectively.
In each case, one component is the unknot, while the other is isotopic to the unknot as shown in Figures \ref{example_2607_(0 0)_+0_1_unkont} and \ref{example_2607_(0 0)_-0_0_unknot}, respectively.

	\begin{figure}[H]
		\centering
		\includegraphics[width=0.65\textwidth]{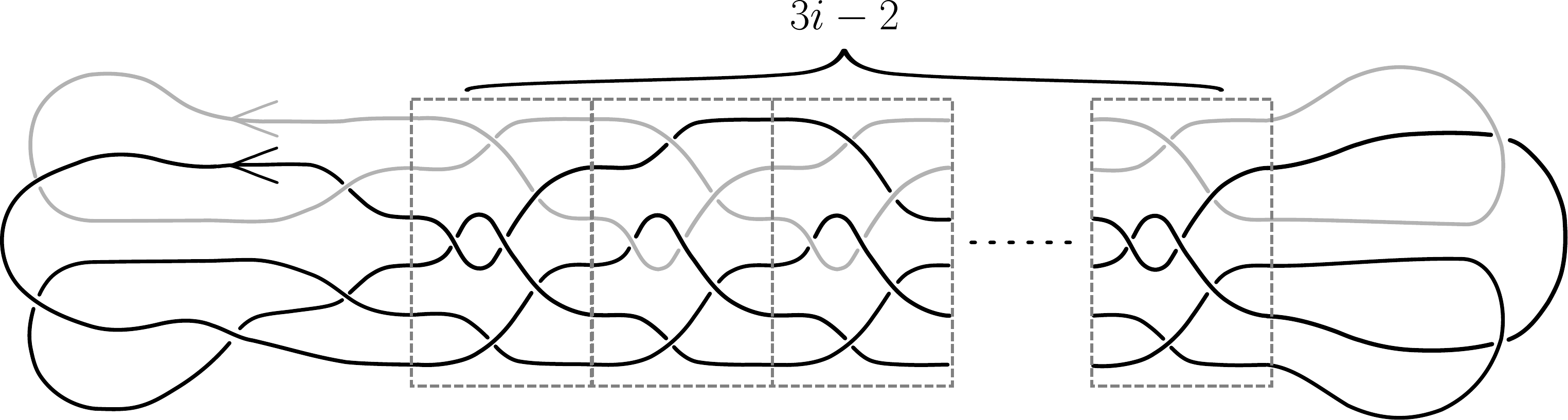}
		\caption{$K_i^1(0)_{+ 0}$} \label{example_2607_(0 0)_+0_1}
	\end{figure}
	\begin{figure}[H]
		\centering
		\includegraphics[width=0.6\textwidth]{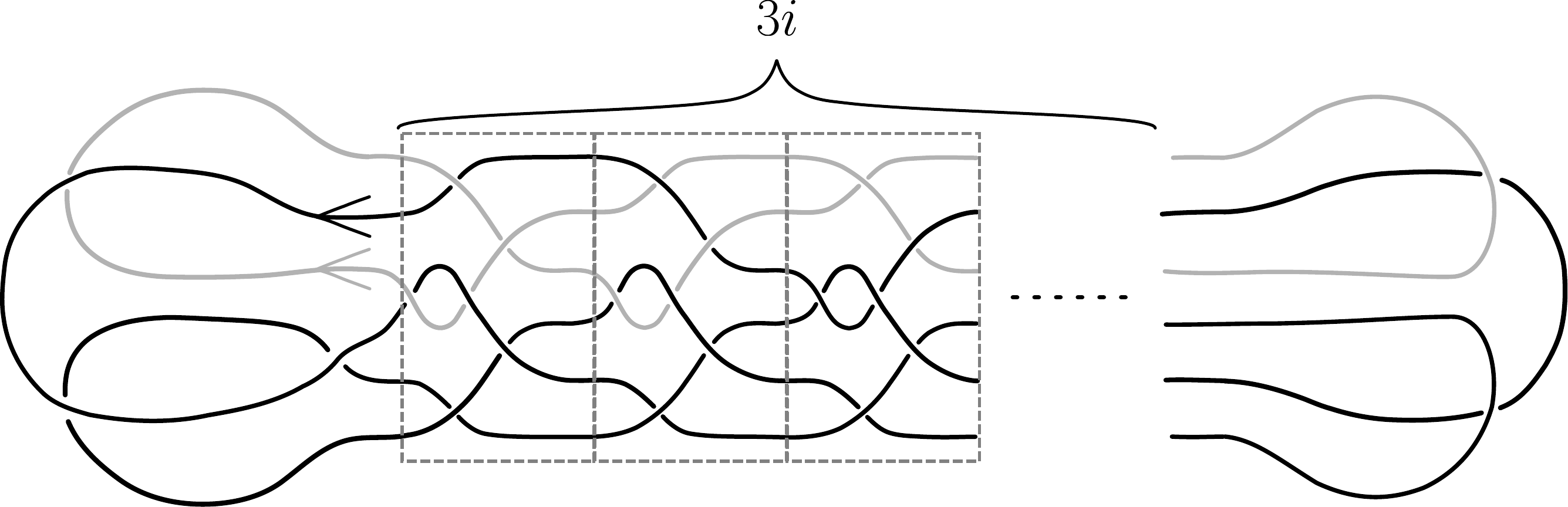}
		\caption{$K_i^0(0)_{- 0}$} \label{example_2607_(0 0)_-0_0}
		\vspace{2em}
		\centering
		\includegraphics[width=1.0\textwidth]{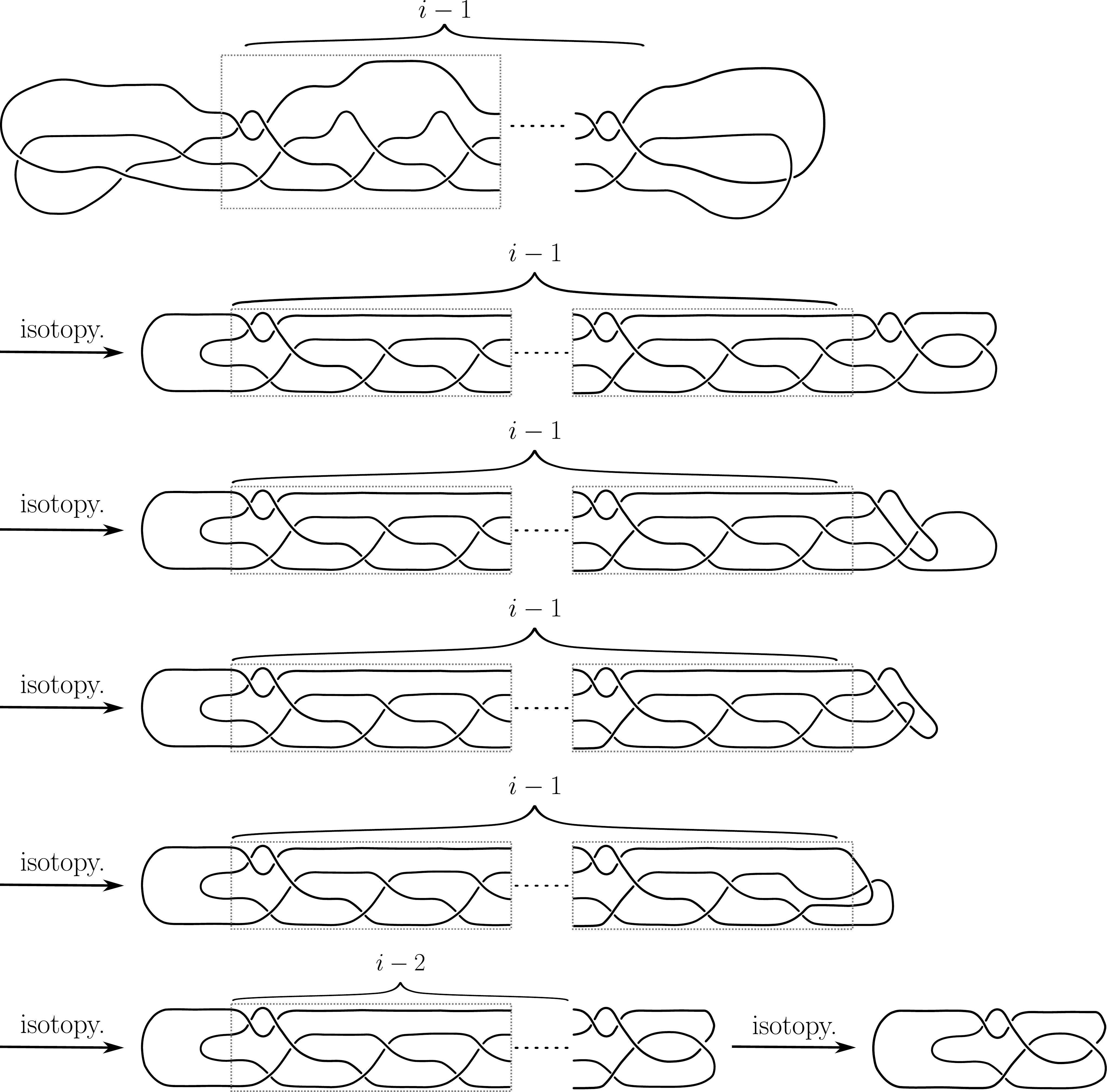}
		\caption{A deformation to make one component unknot.} \label{example_2607_(0 0)_+0_1_unkont}
	\end{figure}
	\begin{figure}[H]
		\centering
		\includegraphics[width=1.1\textwidth]{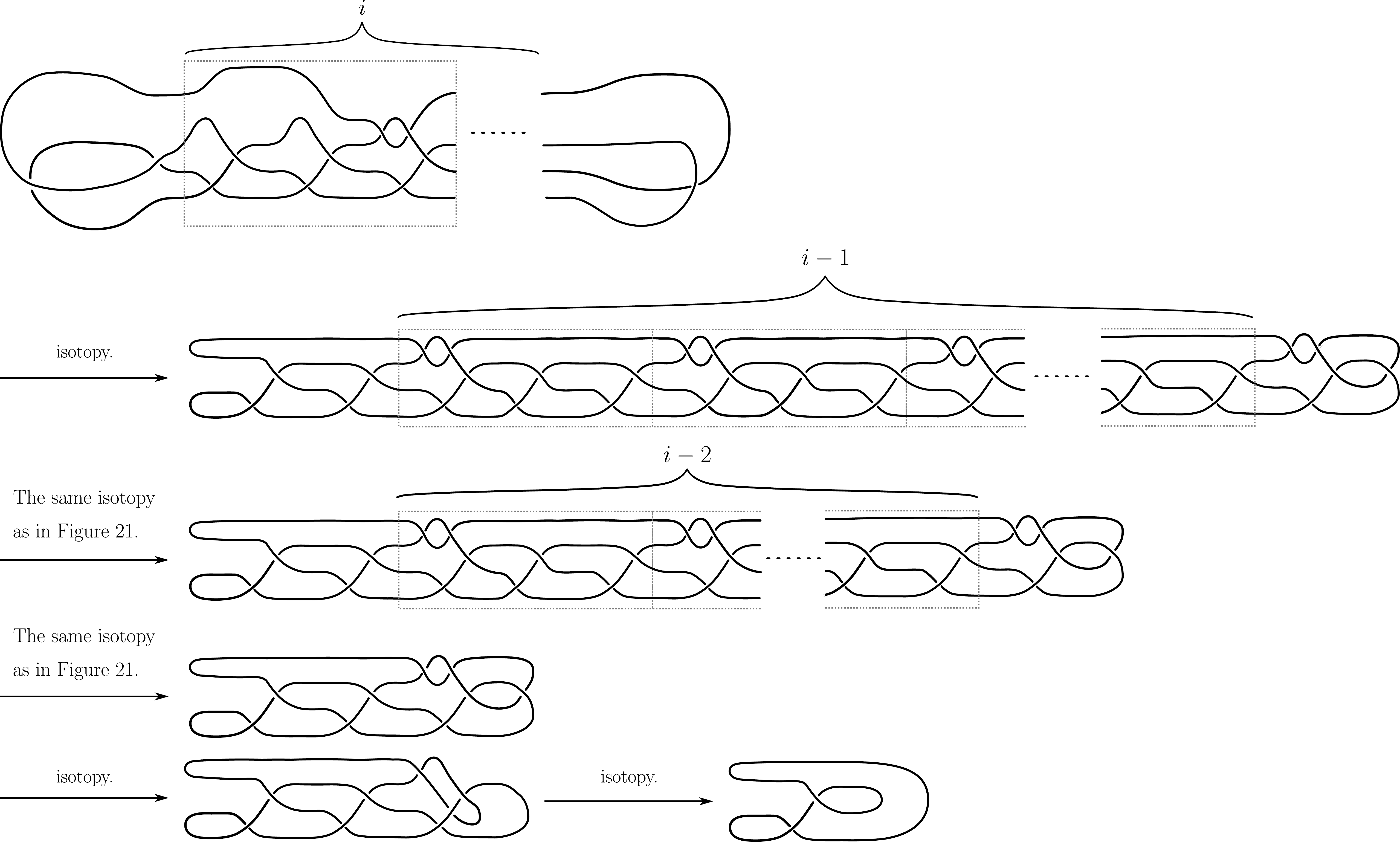}
		\caption{A deformation to make one component unknot.} \label{example_2607_(0 0)_-0_0_unknot}
	\end{figure}
The linking numbers of $K_i^1(0)_{+ 0}$ and $K_i^0(0)_{- 0}$ are $-2i+2$ and $-2i-1$, respectively.
Thus, by a direct computation, we obtain
\begin{align}
	P_0(K_i^1(0)_{+ 0}) &= v^{-4i+4} (v^{-1}-v), \label{K_i^1(0)_{+ 0}} \\
	P_0(K_i^0(0)_{- 0}) &= v^{-4i-2} (v^{-1}-v) \label{K_i^0(0)_{- 0}}
\end{align}
and by Equations \eqref{K_i^1(1)_{+ 0}}--\eqref{K_i^0(0)_{- 0}}, we prove Lemma \ref{pm0}.
\end{proof}

\begin{lemma} \label{mp0}
For any $i \in \mathbb{N}$, we have $P_0(K_i^1(1)_{0 -}) \not= P_0(K_i^1(0)_{0 -})$ and $P_0(K_i^0(1)_{0 +}) \not= P_0(K_i^0(0)_{0 +})$.
\end{lemma}

\begin{proof}[Proof of Lemma \ref{mp0}]
Similar arguments apply to Lemma \ref{pm0}.
First, we consider $P_0(K_i^1(1)_{0 -})$ and $P_0(K_i^0(1)_{0 +})$.
$K_i^1(1)_{0 -}$ is a 2-component link consisting of the unknot and the $(2i-1)$- twist knot, as shown in Figure \ref{example_2607_(1 -1)_0-_1}.
	\begin{figure}[H]
		\centering
		\includegraphics[width=0.65\textwidth]{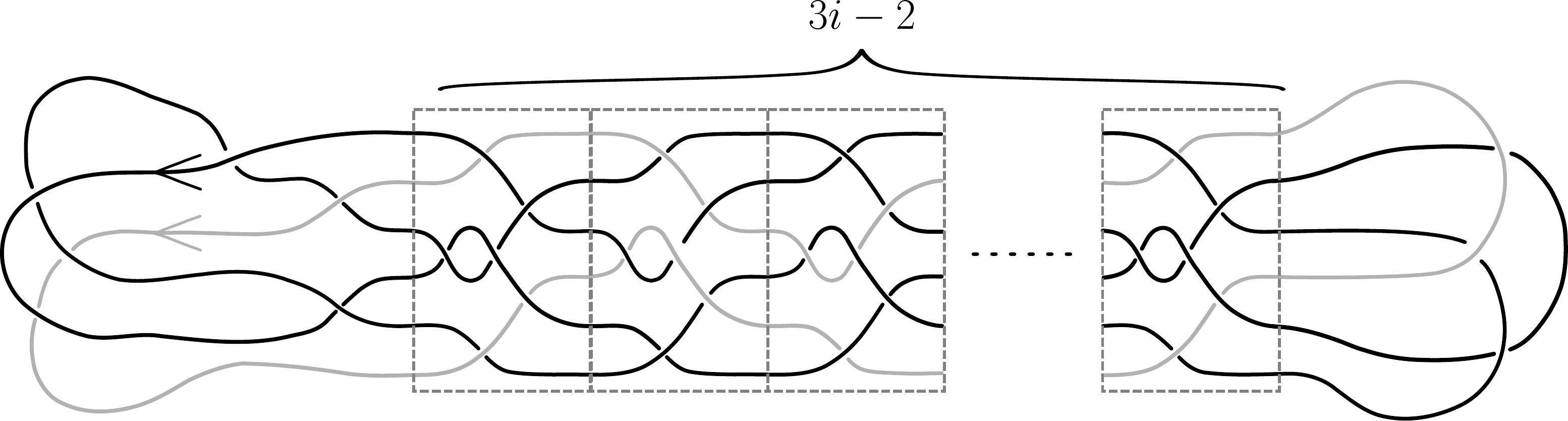}
		\caption{$K_i^1(1)_{0 -}$} \label{example_2607_(1 -1)_0-_1}
	\end{figure}
Similarly, $K_i^0(1)_{0 +}$ is a 2-component link consisting of the unknot and $(2i-1)$- twist knot, as shown in Figures \ref{example_2607_(1 -1)_0+_0}.
	\begin{figure}[H]
		\centering
		\includegraphics[width=0.65\textwidth]{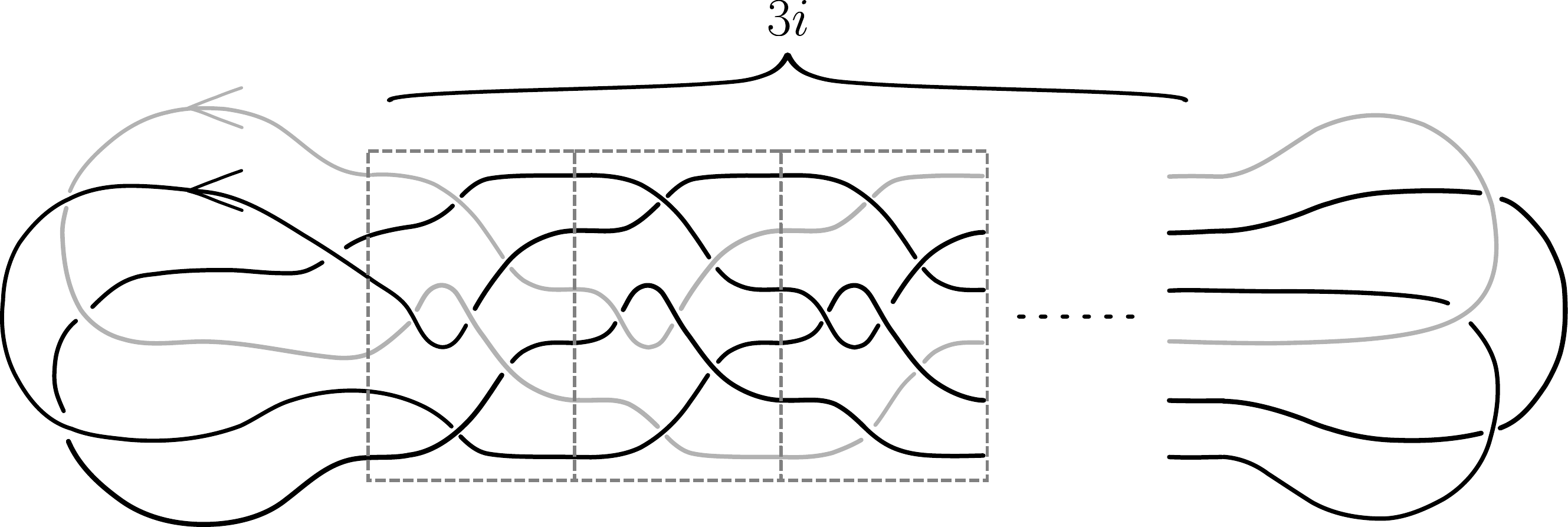}
		\caption{$K_i^0(1)_{0 +}$} \label{example_2607_(1 -1)_0+_0}
	\end{figure}
Moreover, the linking numbers of $K_i^1(1)_{0 -}$ and $K_i^0(1)_{0 +}$ are $-2i+2$ and $-2i+1$, respectively.
Note that, by direct calculation, the $0$-th coefficient polynomial of the HOMFLY polynomial for $(2i-1)$-twist knot is $-v^{-2i-2}+v^{-2i}+v^{-2}$.
Thus, by a direct calculation, we obtain
\begin{align}
	P_0(K_i^1(1)_{0 -}) &= v^{-4i+4} (v^{-1}-v) (-v^{-2i-2}+v^{-2i}+v^{-2}) \label{K_i^1(1)_{0 -}}, \\
	P_0(K_i^0(1)_{0 +}) &= v^{-4i+2} (v^{-1}-v) (-v^{-2i-2}+v^{-2i}+v^{-2}) \label{K_i^0(1)_{0 +}}.
\end{align}

Next, we consider $P_0(K_i^1(0)_{0 -})$ and $P_0(K_i^0(0)_{0 +})$.
As shown in Figure \ref{example_2607_(0 0)_0-_1}, $K_i^1(0)_{0 -}$ consists of the unknot and the knot, say $J_i^1$, depicted in Figure \ref{example_2607_(0 0)_0-_J1}.
	\begin{figure}[H]
		\centering
		\includegraphics[width=0.65\textwidth]{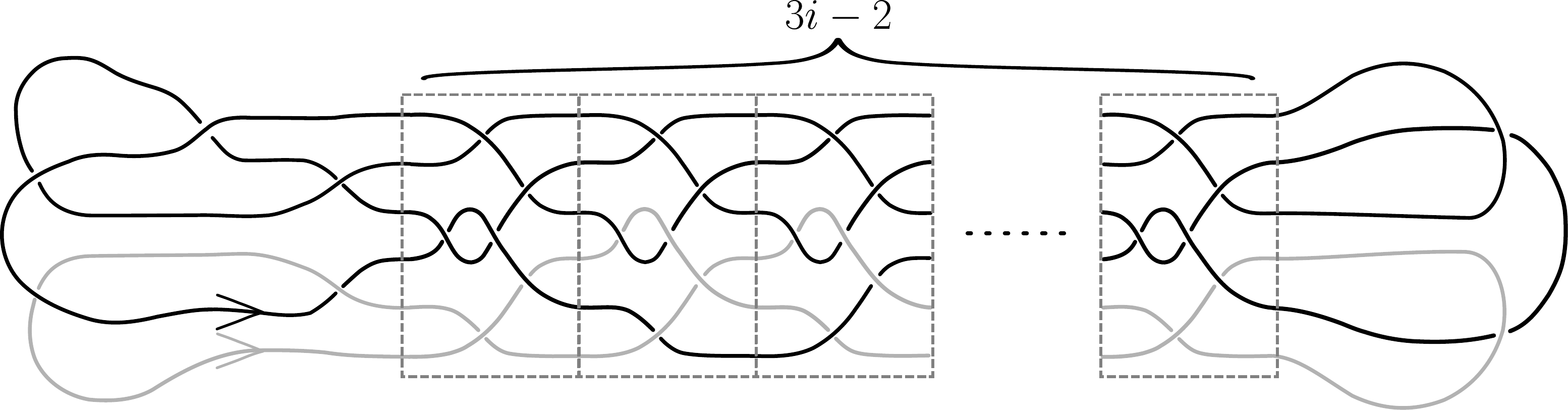}
		\caption{$K_i^1(0)_{0 -}$} \label{example_2607_(0 0)_0-_1}
	\end{figure}
	\begin{figure}[H]
		\centering
		\includegraphics[width=0.65\textwidth]{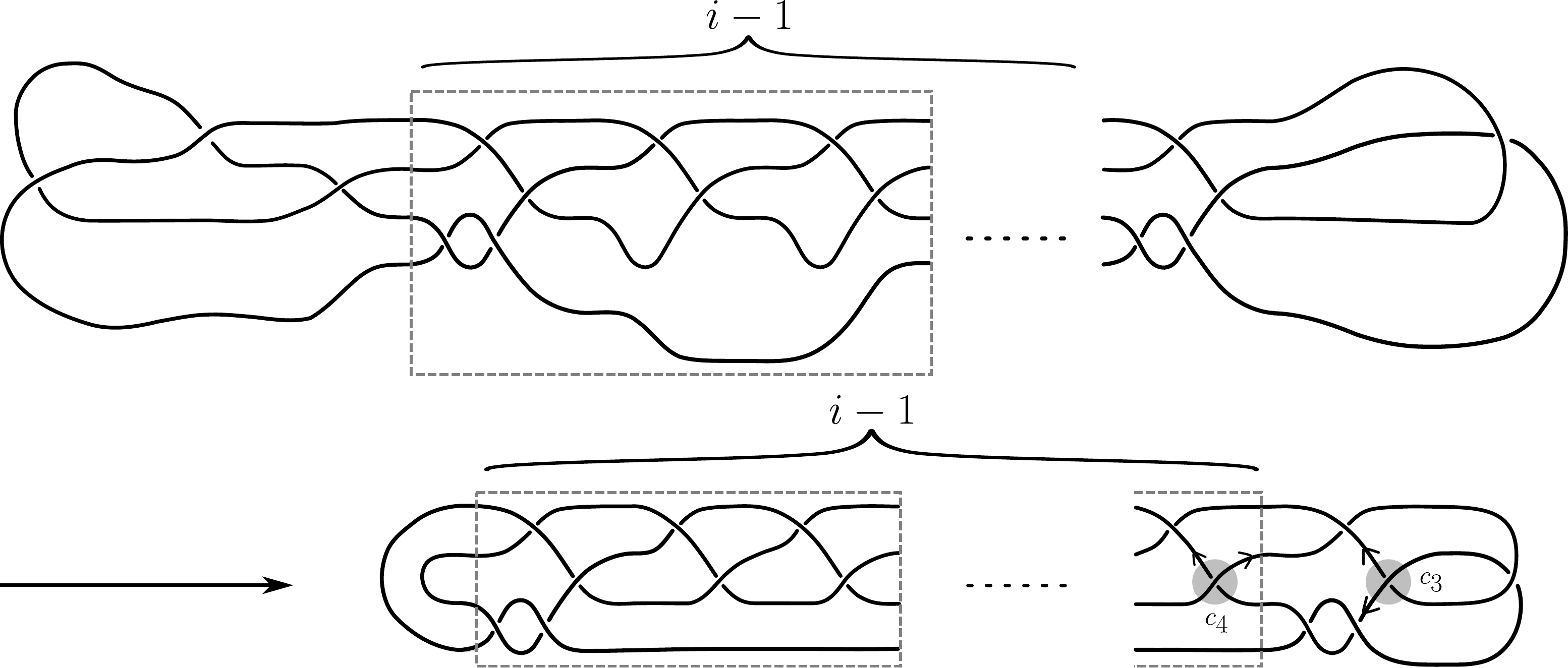}
		\caption{$J_i^1$} \label{example_2607_(0 0)_0-_J1}
	\end{figure}
Then, the following claim holds.
\begin{claim} \label{J_i^1}
$P_0(J_i^1) = -v^{-2i-4}+v^{-2i}+v^{-4}$.
\end{claim}

\begin{proof}[Proof of Claim \ref{J_i^1}]
We consider a crossing change and a smoothing at the two crossings $c_3$ and $c_4$ shown in Figure \ref{example_2607_(0 0)_0-_J1}. 
Let $(J_i^1)_{+}$ and $(J_i^1)_{0}$ be the knots obtained from $J_i^1$ by changing $c_3$ to a positive crossing and by smoothing $c_3$, respectively.
First, $(J_i^1)_{0}$ is a 2-component link whose components are both unknots, with linking number $-i-1$, as shown in Figure \ref{example_2607_(0 0)_0-_J1_0}.
	\begin{figure}[h]
		\centering
		\includegraphics[width=0.65\textwidth]{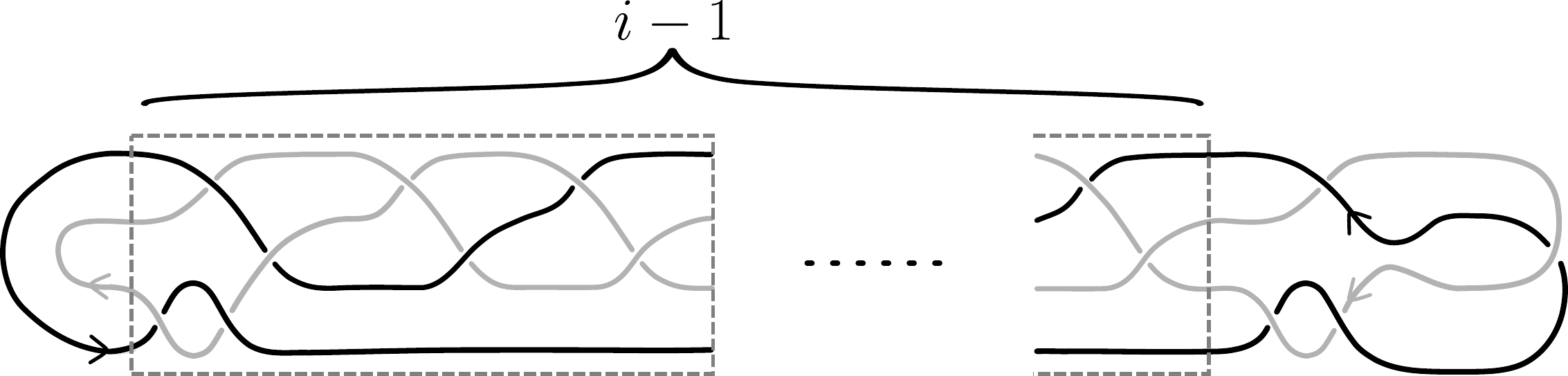}
		\caption{$(J_i^1)_{0}$} \label{example_2607_(0 0)_0-_J1_0}
	\end{figure}
Moreover, we apply a crossing change and a smoothing to $(J_i^1)_{+}$ at $c_4$.
Let $(J_i^1)_{+ -}$ and $(J_i^1)_{+ 0}$ be the knots obtained from $(J_i^1)_{+}$ by changing $c_4$ to a negative crossing and by smoothing  $c_4$, respectively.
Then, $(J_i^1)_{+ -}$ is isotopic to $J_{i-1}^1$.
$(J_i^1)_{+0}$ is a 2-component link whose components are both unknots, with linking number $-i+1$ as shown in Figure \ref{example_2607_(0 0)_0-_J1_00_2}.
	\begin{figure}[H]
		\centering
		\includegraphics[width=0.65\textwidth]{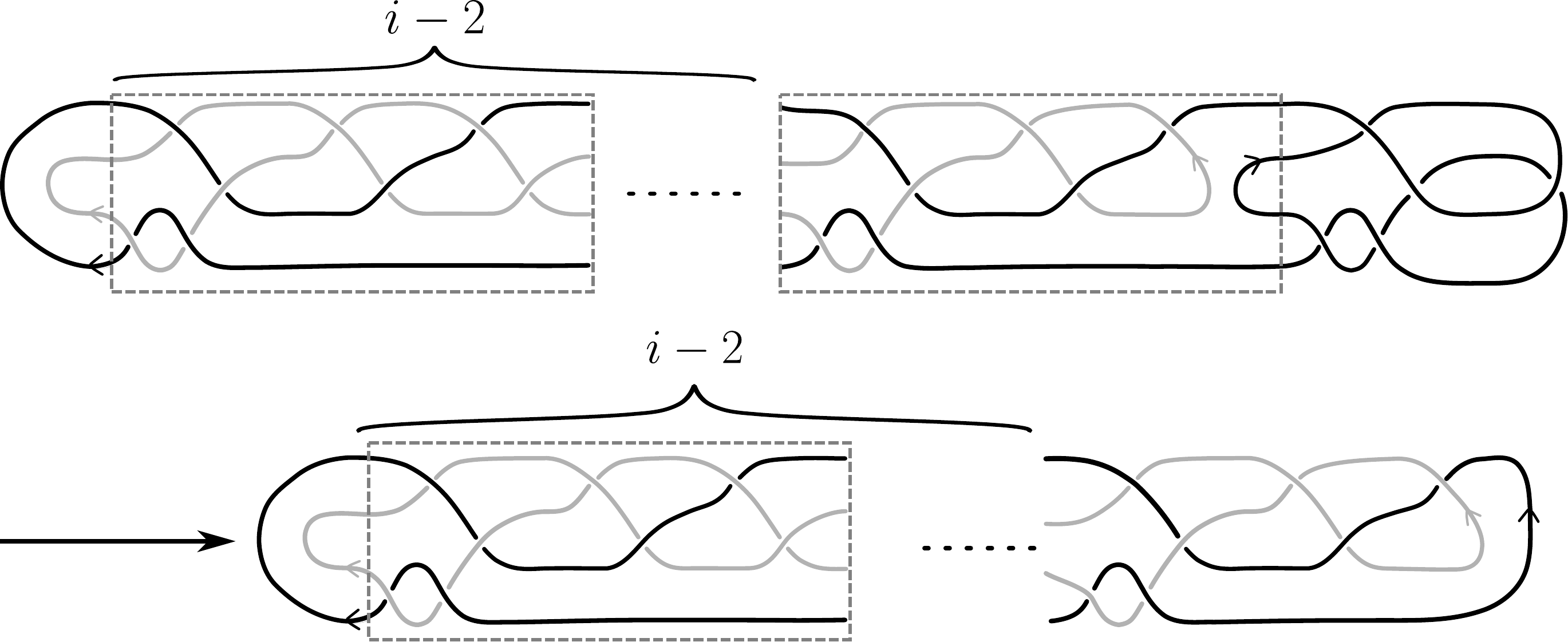}
		\caption{$(J_i^1)_{+0}$} \label{example_2607_(0 0)_0-_J1_00_2}
	\end{figure}
From this, we obtain
\begin{align*}
	P_0(J_i^1) = P_0(J_{i-1}^1) + v^{-1}(-P_0((J_i^1)_{0})+P_0((J_i^1)_{+0})).
\end{align*}

Note that $J_0^1$ is the unknot.
By calculating inductively with this, we can complete Claim \ref{J_i^1}.
\end{proof}

As shown in Figure \ref{example_2607_(0 0)_0+_0}, $K_i^0(0)_{0 +}$ consists of the unknot and the knot, say $J_i^0$, depicted in Figure \ref{example_2607_(0 0)_0+_0_J0}.
	\begin{figure}[H]
		\centering
		\includegraphics[width=0.65\textwidth]{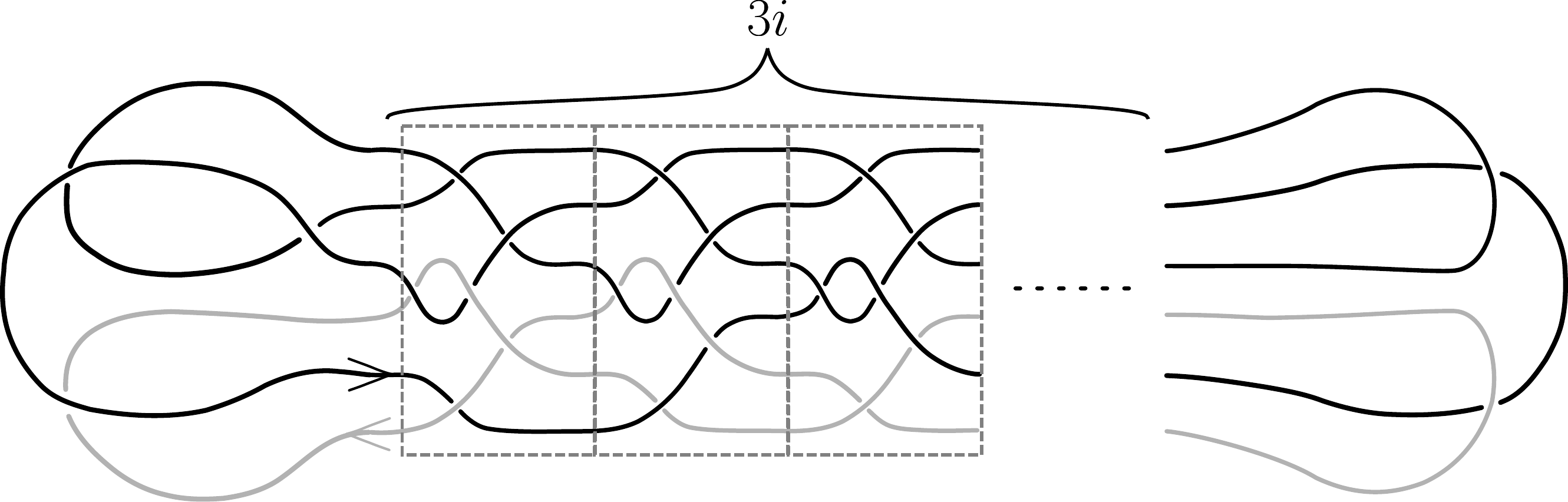}
		\caption{$K_i^0(0)_{0 +}$} \label{example_2607_(0 0)_0+_0}
		\vspace{1em}
		\centering
		\includegraphics[width=0.7\textwidth]{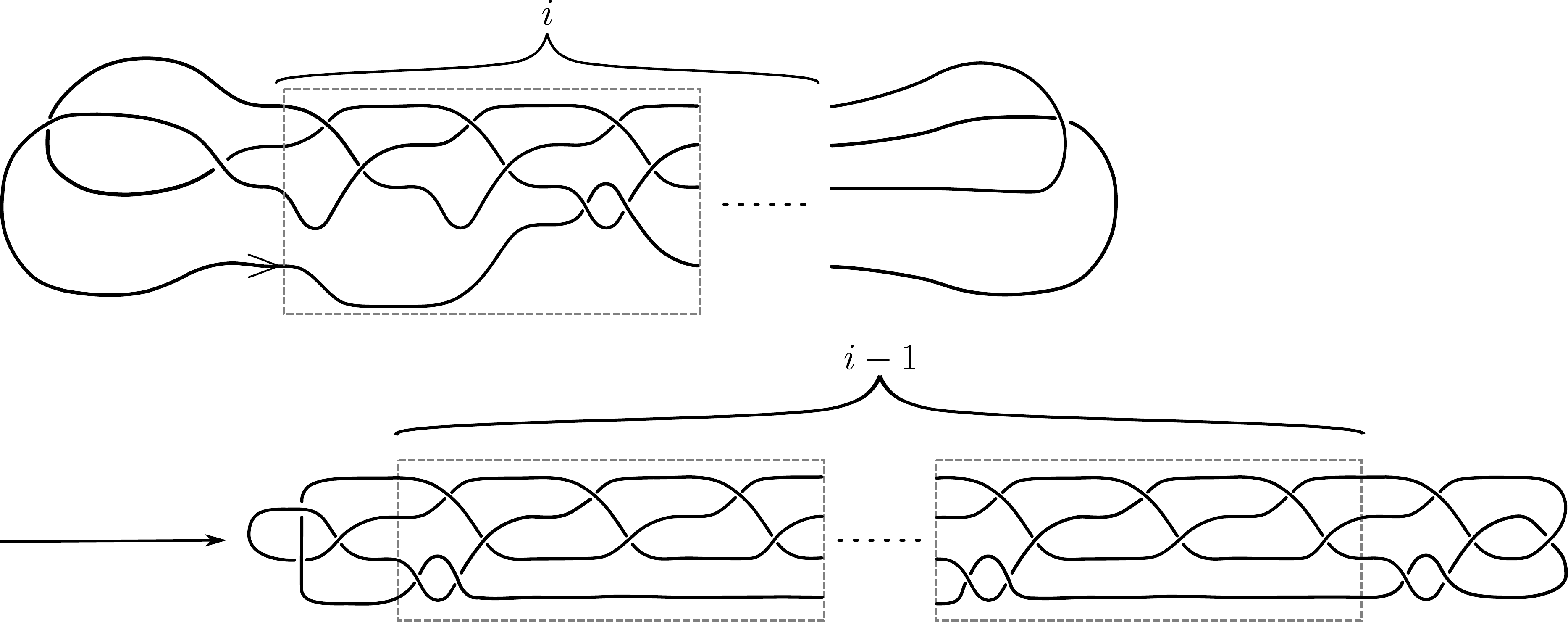}
		\caption{$J_i^0$} \label{example_2607_(0 0)_0+_0_J0}
	\end{figure}
An inductive calculation similar to Claim \ref{J_i^1} gives
\begin{align*}
	P_0(J_i^0) = -v^{-2i-2}+v^{-2i+2}+v^{-2}+1-v^2.
\end{align*}

Moreover, the linking numbers of $K_i^1(0)_{0 -}$ and $K_i^0(0)_{0 +}$ are $-2i+2$ and $-2i+1$, respectively.
Thus, by a direct calculation, we obtain
\begin{align}
	P_0(K_i^1(0)_{0 -}) &= v^{-4i+4} (v^{-1}-v) (-v^{-2i-4}+v^{-2i}+v^{-4}) \label{K_i^1(0)_{0 -}}, \\
	P_0(K_i^0(0)_{0 +}) &= v^{-4i+2} (v^{-1}-v) (-v^{-2i-2}+v^{-2i+2}+v^{-2}+1-v^2) \label{K_i^0(0)_{0 +}}
\end{align}
and by Equations \eqref{K_i^1(1)_{0 -}}--\eqref{K_i^0(0)_{0 +}}, we prove Lemma \ref{mp0}.
\end{proof}

By Lemma \ref{pm0} and \ref{mp0}, we show
\begin{align}
	v^{-1} P_0(K_{l}^1(1)_{+0}) -v P_0(K_{l}^1(1)_{0-}) &= v^{-4l+4} (v^{-1}-v)^2 (2v^{-2l}-1), \label{P_0(K_i^1(1))} \\
	v^{-1} P_0(K_{l}^1(0)_{+0}) -v P_0(K_{l}^1(0)_{0-}) &= v^{-4l+4} (v^{-1}-v)^2 (v^{-2l-2}-v^{-2l}-v^{-2}), \label{P_0(K_i^1(0))} \\
	- v P_0(K_{l}^0(1)_{-0}) + v^{-1} P_0(K_{l}^0(1)_{0+}) &= -v^{-4l} (v^{-1}-v)^2 (2v^{-2l}-1), \label{P_0(K_i^0(1))} \\
	- v P_0(K_{l}^0(0)_{-0}) + v^{-1} P_0(K_{l}^0(0)_{0+}) &= -v^{-4l-2} (v^{-1}-v)^2 (v^{-2l+2}+v^{-2l+4}-v^4) \label{P_0(K_i^0(0))}.
\end{align}
Using \eqref{P_0(K_i^1(1))}-- \eqref{P_0(K_i^0(0))}, direct calculation shows
\begin{align}
	\begin{split}
			 &\sum^{i}_{l=1} v^{-4(i-l)} (v^{-1} P_0(K_{l}^1(1)_{+0}) -v P_0(K_{l}^1(1)_{0-}) \\ 
			 &+v^{-2} \sum^{i-1}_{l=1} v^{-4(i-1-l)} (- v P_0(K_{l}^0(1)_{-0}) + v^{-1} P_0(K_{l}^0(1)_{0+})) \\%
			 &= v^{-4i+2} (v^{-1}-v)^2 (i+1-iv^2)
	\end{split} \label{LHSof(4.2)}
\end{align}
and
\begin{align}
	\begin{split}
			 &\sum^{i}_{l=1} v^{-4(i-l)} (v^{-1} P_0(K_{l}^1(0)_{+0}) -v P_0(K_{l}^1(0)_{0-}) \\ 
			 &+v^{-2} \sum^{i-1}_{l=1} v^{-4(i-1-l)} (- v P_0(K_{l}^0(0)_{-0}) + v^{-1} P_0(K_{l}^0(0)_{0+})) \\
			 &= v^{-4i+4} (v^{-1}-v)^2 (i-1+v^{-2i-2}+v^{-2i}-iv^{-2}).
	\end{split} \label{RHSof(4.2)}
\end{align}
The minimal degree of \eqref{LHSof(4.2)} and \eqref{RHSof(4.2)} are equal to $-4i$ and  $-6i$, respectively.
Similarly, 
\begin{align} 
	\begin{split}
			 &v^{-2} \sum^{i}_{l=1} v^{-4(i-l)} (v^{-1} P_0(K_{l}^1(1)_{+0}) -v P_0(K_{l}^1(1)_{0-})) \\
			 &+\sum^{i}_{l=1} v^{-4(i-l)} (- v P_0(K_{l}^0(1)_{-0}) + v^{-1} P_0(K_{l}^0(1)_{0+})) \\
			 &= -v^{-4i+2} (v^{-1}-v)^2 (i+2v^{-2i-2}-(i+2)v^{-2}) \\
	\end{split} \label{LHSof(4.3)}
\end{align}
and
\begin{align}
	\begin{split}
			&v^{-2} \sum^{i}_{l=1} v^{-4(i-l)} (v^{-1} P_0(K_{l}^1(0)_{+0}) -v P_0(K_{l}^1(0)_{0-})) \\
			 &+\sum^{i}_{l=1} v^{-4(i-l)} (- v P_0(K_{l}^0(0)_{-0}) + v^{-1} P_0(K_{l}^0(0)_{0+})) \\
			 &= v^{-4i+2} (v^{-1}-v)^2 i(1-v^{-2}).
	\end{split} \label{RHSof(4.3)}
\end{align}
The minimal degree of \eqref{LHSof(4.3)} and \eqref{RHSof(4.3)} are equal to $-6i-2$ and  $-4i-2$, respectively.
Thus we obtain \eqref{deg of P_0(K_i^1(n))} and \eqref{deg of P_0(K_i^0(n))}.
Hence, the proof of Proposition \ref{K_{i}^j(n,-n)} is complete.

\section*{Acknowledgement}
I would like to express my appreciation to Professor Yoshiyuki Yokota for his constructive suggestions and thoughtful guidance.
This work was supported by JST SPRING, Grant Number JPMJSP2156.



\bigskip
Yuki Matsushima

\address{
  Department of Mathematical Sciences, Tokyo Metropolitan
University, Minami-Osawa Hachioji, Tokyo, 192-0397, Japan.
}

\textit{E-mail address}: \href{mailto:matsushima-yuki@ed.tmu.ac.jp}{\texttt{matsushima-yuki@ed.tmu.ac.jp}}

\end{document}